\documentclass[11pt]{article}

\usepackage[margin=1in]{geometry}
\usepackage[authoryear,longnamesfirst]{natbib}
\usepackage{booktabs}
\usepackage{array}
\usepackage{graphicx}
\usepackage{xcolor}
\usepackage{url}
\usepackage[hidelinks]{hyperref}
\usepackage{amsmath,amssymb,amsthm,bm}

\newcommand{\reals}{\mathbb{R}}
\newcommand{\E}{\mathbb{E}}
\newcommand{\Var}{\operatorname{Var}}

\newcommand{\Prob}{\mathbb{P}}
\newcommand{\sign}{\operatorname{sign}}

\newcommand{\argmin}{\operatorname*{arg\,min}}

\newcommand{\indic}{\mathbf{1}}
\newcommand{\todist}{\xrightarrow{d}}

\newcommand{\genr}{\varphi}                       
\newcommand{\Dsq}[1]{D_{#1}^{2}}                  
\newcommand{\Done}{\Dsq{1}}
\newcommand{\Dzero}{\Dsq{0}}
\newcommand{\Sig}{\Sigma}
\newcommand{\LLR}{\Lambda}                          
\newcommand{\LLRhat}{\widehat{\Lambda}}
\newcommand{\Rstar}{R^{\star}}                      
\newcommand{\bx}{\bm{x}}
\newcommand{\bmu}{\bm{\mu}}
\newcommand{\bbeta}{\bm{\beta}}
\newcommand{\bZ}{\bm{Z}}
\newcommand{\T}{\mathsf{T}}                          
\newcommand{\eqd}{\stackrel{d}{=}}
\newcommand{\sieve}{m}                               

\theoremstyle{plain}
\newtheorem{theorem}{Theorem}
\newtheorem{lemma}[theorem]{Lemma}
\newtheorem{proposition}[theorem]{Proposition}
\newtheorem{corollary}[theorem]{Corollary}
\theoremstyle{definition}
\newtheorem{assumption}{Assumption}

\hypersetup{
  pdftitle={Closed-form fractional radial links for elliptical Mahalanobis discriminant analysis},
  pdfauthor={Serhii Zabolotnii},
  pdfkeywords={elliptical distributions, discriminant analysis,
    Mahalanobis distance, generalized additive model, radial link,
    Bayes-optimality, heavy tails}
}

\title{Closed-form fractional radial links for elliptical Mahalanobis discriminant analysis}

\author{Serhii Zabolotnii\\
\small Department of Information, Multimedia Technologies and Design,\\
\small Cherkasy State Business College, Cherkasy 18028, Ukraine\\
\small State Scientific Research Institute of Armament and Military Equipment\\
\small Testing and Certification, Cherkasy, Ukraine\\
\small Department of Cybernetics and Applied Mathematics,\\
\small Uzhhorod University, Uzhhorod, Ukraine\\
\small ORCID: \href{https://orcid.org/0000-0003-0242-2234}{0000-0003-0242-2234}\\
\small \texttt{zabolotnii.serhii@csbc.edu.ua}}
\date{2026}

\begin{document}
\maketitle

\begin{abstract}
\noindent
We study binary classification when the two classes are elliptical with a
shared radial generator $g$. The log-likelihood ratio is then an additive
function of the two squared Mahalanobis radii, with radial link $\genr=\log g$.
For Gaussian classes $\genr$ is affine and the rule reduces to quadratic
discriminant analysis (QDA); otherwise the Bayes link is not the identity.
Instead of fitting this link with a spline smoother, we derive its family from
the within-class law of the radius and estimate a finite fractional-power
projection. We prove that the link is identifiable from the radius law, that the
plug-in estimator is $\sqrt{n}$-consistent and asymptotically normal under
finite-moment conditions, and that the induced classifier is asymptotically
Bayes-optimal in an iterated sieve limit. For $d\ge2$ and non-nested class
parameters, the identity link suffices only for affine (Gaussian) generators.
The bridge, GAM membership and the affine direction of this dichotomy are
verified in Lean~4. We compare against the global Mahalanobis-GAM of Ghosh et
al.\ (2025), re-implemented with \texttt{mgcv} REML splines at equal input
budget. On two UCI benchmarks the derived link ties with the fitted GAM
(bootstrap CIs $[-0.005,+0.003]$ and $[-0.008,+0.019]$). On six
financial series from three asset classes, with validation that respects
temporal dependence, it is never significantly worse than the fitted GAM: better
on CAD/USD and tied on the other five. It beats QDA
on the three heaviest-tailed series (oil $[+0.023,+0.069]$, S\&P~500
$[+0.039,+0.125]$, JPY/USD $[+0.009,+0.047]$) and ties elsewhere. A rate
simulation with a closed-form ground truth reproduces the $\sqrt{n}$ rate and,
under covariance heterogeneity, QDA's risk floor against the derived link's
vanishing excess risk. The derived link thus matches a tuned GAM without
smoothing-parameter selection and improves on QDA where the generator departs
from the Gaussian.
\end{abstract}

\noindent\textbf{Highlights:}
\begin{itemize}
\item Elliptical Bayes rule is additive in the two Mahalanobis radii
\item In two or more dimensions the identity link suffices only for Gaussians
\item Derived fractional-power sieve avoids fitting a spline GAM
\item The estimator is root-n consistent and Bayes-optimal in an iterated limit
\item It ties with the fitted GAM and beats QDA on the heaviest-tailed series
\end{itemize}



\section{Introduction}
\label{sec:intro}

Quadratic discriminant analysis (QDA) classifies a feature $\bx\in\reals^d$ by
comparing two squared Mahalanobis radii,
$\Dsq{c}(\bx)=(\bx-\bmu_c)^{\T}\Sig_c^{-1}(\bx-\bmu_c)$, one per class. It is
the Bayes rule when the classes are Gaussian. When they are not, QDA is no
longer optimal, and a natural fix has gained currency: keep the Mahalanobis
radii as features but feed them to a flexible nonparametric classifier rather
than to the fixed quadratic. \citet{ghosh2025} do exactly this, combining a
global and a local Mahalanobis distance per class inside a generalized additive
model (GAM), and report strong performance across multivariate benchmarks.
This raises a structural question that their construction leaves implicit: when
the classes are elliptical but non-Gaussian, \emph{what function of the radii is
the Bayes rule}, and is fitting it nonparametrically the right thing to do, or
can it be written down?

This paper answers both. For two elliptical class-conditional densities sharing
a radial generator $g$ --- so that $\log p_c(\bx)=\genr(\Dsq{c}(\bx))-\tfrac12
\log|\Sig_c|+\text{const}$ with $\genr=\log g$ --- the log-likelihood ratio is,
identically,
\begin{equation}
\label{eq:bridge-intro}
\LLR(\bx)=\genr(\Done(\bx))-\genr(\Dzero(\bx))+\tfrac12\log\frac{|\Sig_0|}{|\Sig_1|}
+\log\frac{\pi_1}{\pi_0}.
\end{equation}
The posterior log-odds are therefore additive in the two radii, with the radial
link $\genr$ as component function: a GAM. This structure is the elliptic-class
result of \citet{ghosh2025} (their Theorem~1, stated in the Mahalanobis distances
$\delta_c=\sqrt{\Dsq{c}}$). We use the squared-radius form throughout and take
their observation as the starting point, not as a contribution. The Gaussian case
is degenerate. There $g(t)\propto e^{-t/2}$, so $\genr$ is affine,
\eqref{eq:bridge-intro} is affine in the radii, and the Bayes rule is QDA: the
Mahalanobis-GAM classifier and the Gaussian discriminant are one object under
two names. Equation~\eqref{eq:bridge-intro} says something new only for
non-Gaussian carriers, where $\genr$ is non-affine and the link is not the
identity.

The non-identity link can then be \emph{fitted}, as a spline GAM
\citep{ghosh2025}, or \emph{derived}; we take the second route. The fractional-power
radial basis comes from the stochastic-polynomial apparatus of
\citet{kunchenko2002}, namely the polynomial maximization method (PMM) and its
parametrically-adaptive transition polynomial \citep{zabolotnii2026patp}. Its
coefficients are estimated by logistic risk minimization on the per-class radial
features. The captured-fraction principle \citep{zabolotnii2026paper4} explains
why this low-dimensional basis is the right target. Deriving the link replaces a
two-dimensional nonparametric additive fit by a one-dimensional projection of known
functional form. We test whether this yields parity with a tuned GAM at lower
tuning cost, and an advantage over QDA under heavy tails. In single-index terms
\citep{ichimura1993,hardle1993}, where the index direction and the link are both
estimated nonparametrically, the elliptical structure supplies the index (the
Mahalanobis radius) and the form of the link (from the generator). The problem
becomes parametric, which is why a $\sqrt{n}$ rate is attainable below.

\paragraph{Contributions.}
\begin{enumerate}
\item[\textbf{(C1)}] \emph{Structure.} Taking the elliptic GAM structure of
\citet{ghosh2025} as given, we show that the identity link represents
\emph{exactly} the affine log-generators. By normalizability, the Gaussian is the
only shared-generator elliptical family whose log-odds lie in the identity-link
span of the two radii. For $d\ge2$ and non-nested class parameters, every analytic
non-affine generator places the log-odds outside that span
(Proposition~\ref{prop:dichotomy}). The bridge, GAM membership, the affine
direction of this dichotomy and a Student-$t$ witness are machine-checked in
Lean~4; the converse direction is proved on paper. The dichotomy, not the bridge,
is the new content.

\item[\textbf{(C2)}] \emph{Identifiability.} The radial link $\genr$ is
identifiable from the within-class law of the Mahalanobis radius, with an
explicit inversion (Lemma~\ref{lem:ident}). This makes ``estimate the link'' a
well-posed problem and dictates a basis in powers of the radius.

\item[\textbf{(C3)}] \emph{Estimation theory.} The plug-in link estimator
built from a fractional-power sieve is $\sqrt{n}$-consistent and asymptotically
normal (Theorem~\ref{thm:clt}); the induced plug-in classifier is
asymptotically Bayes-optimal in the iterated sieve limit whenever the radial
links lie in the closure of the power sieve (Theorems~\ref{thm:approx}
and~\ref{thm:bayes}). A closed-form rate simulation corroborates both
(Section~\ref{sec:exp}). This is the main contribution beyond the near-definitional
\eqref{eq:bridge-intro}: a derived link with statistical guarantees.

\item[\textbf{(C4)}] \emph{Evidence and scope.} We re-implement the global
Mahalanobis-GAM of \citet{ghosh2025} with penalized-spline smoothing
(\texttt{mgcv}, REML). At equal input budget the derived link ties with this
fitted link on real benchmarks, without smoothing selection. On six financial
series from three asset classes, under validation that respects temporal
dependence, it is never significantly worse than the fitted GAM. It beats QDA on
the heaviest-tailed series and ties on a light-tailed control, so the advantage
tracks tail-heaviness. A direct test rejects ellipticity of the real series
through asymmetry, so these results also measure robustness. The advantage grows
with covariance heterogeneity and tail-heaviness and vanishes toward the Gaussian
limit: the method adapts to the regime rather than serving a heavy-tail niche.
\end{enumerate}

The structural layer (C1) refines a Gaussian unification \citep{zabolotnii2026paper4}
that is tautological when the carrier is Gaussian; the elliptical case is the first
in which the statement is falsifiable and useful. The weight of the paper is on
the estimation theory (C3) and the benchmark (C4): a derived radial link that is
consistent and asymptotically Bayes-optimal, with its structural algebra partly
machine-checked, and on par with the standard fitted alternative at a lower tuning
cost.

\section{Background and related work}
\label{sec:background}

\subsection{Elliptical distributions and Mahalanobis radii}
A random vector $\bx\in\reals^d$ is elliptically distributed with location
$\bmu$, scatter $\Sig\succ0$ and radial generator $g$ if its density is
$p(\bx)=|\Sig|^{-1/2}g\big((\bx-\bmu)^{\T}\Sig^{-1}(\bx-\bmu)\big)$
\citep{fang1990,cambanis1981}. The quadratic form
$\Dsq{}(\bx)=(\bx-\bmu)^{\T}\Sig^{-1}(\bx-\bmu)$ is the squared Mahalanobis
radius; the family includes the Gaussian ($g(t)\propto e^{-t/2}$), the
multivariate Student-$t_\nu$ ($g(t)\propto(1+t/\nu)^{-(\nu+d)/2}$) and the
power-exponential laws. The radial decomposition $\bx=\bmu+R\,\Sig^{1/2}\bm{u}$,
with $\bm u$ uniform on the sphere and $R^2\eqd\Dsq{}$, is the standard tool
\citep{cambanis1981} and underlies our identifiability result. The scatter $\Sig$
that defines the radius is itself an estimation target for heavy-tailed
elliptical data. Besides the classical robust estimators (Tyler's $M$-estimator
and related methods, \S\ref{sec:discussion}), there are high-breakdown,
high-efficiency scatter estimators for the full symmetric-elliptical class
\citep{fishbone2024}. Any of them can supply the whitening our plug-in procedure
needs.

Estimating the generator itself is a classical problem.
\citet{liebscher2005} estimates $g$ by kernel smoothing of the transformed radius.
The estimator has a bandwidth and a second parameter that controls the behaviour
near the centre; \citet{ryan2025} analyse how to choose them. \citet{battey2014}
estimate the whole elliptic density by a two-stage finite Gaussian-mixture sieve
on the radial part, at a nonparametric rate. Recent clustering methods carry an
unknown radial generator into the assignment rule
\citep{feng2026clustering,teng2026clusterwise}. Lemma~\ref{lem:ident} is the
population fact behind all of these radius-based estimators, but we use it
differently. We do not estimate $g$ as a density, with its normalizing constant
and smoothing parameters. We estimate only the discriminative contrast
$\genr_1(\Done)-\genr_0(\Dzero)$, by logistic risk on a fixed fractional-power
sieve. At each sieve size this is a parametric problem (Theorem~\ref{thm:clt});
the approximation error is treated separately (Theorem~\ref{thm:approx}).

\subsection{Mahalanobis-distance classification and additive models}
\citet{ghosh2025} made precise the use of per-class Mahalanobis radii as
low-dimensional discriminative summaries. They augment the global radius with a
\emph{local} (nearest-neighbour) Mahalanobis distance and combine both in a GAM
fitted by smoothing splines. This is the closest prior work and our principal
comparator. Their method treats the radius link as an unknown smooth function and
selects the spline smoothness (and, for the local variant, a bandwidth). Our
equal-budget comparison keeps the global per-class radii but uses the elliptical
generator to fix a fractional-power family for the link, so only finitely many
projection coefficients are fitted. Generalized additive models are due to
\citet{hastie1990}. Flexible discriminant analysis by optimal scoring
\citep{hastie1994} recasts discriminant analysis as nonparametric regression; our
construction fixes the transform that such methods fit. What is new is not the
GAM template. For elliptical classes the additive structure and the link are not
modelling choices: \eqref{eq:bridge-intro} forces them, and the link is a known
functional of the generator.

A parallel nonparametric line classifies by a centrality summary instead of a
fitted link. Maximum-depth and depth-versus-depth rules
\citep{ghosh2005depth,li2012} assign each point by its vector of class depths, and
\citet{hubert2017} combine depth \emph{and} distance features in one classifier.
The last is closest in spirit to our use of the radii, since the radius is a
distance-type centrality summary. These methods do not use the elliptical
generator, so there is no link to derive: the summary is an input to a learned
rule. Here the generator dictates which function of the summary is Bayes-optimal.
Classical discriminant analysis \citep{mclachlan2004} is the affine-link special
case of \eqref{eq:bridge-intro}.

\subsection{Known versus estimated links: the single-index connection}
Writing the Bayes rule as a transform of the radii places our problem in the
single-index family, where a response depends on the covariates through a
univariate link applied to a low-dimensional index
\citep{ichimura1993,hardle1993}. That literature faces two coupled unknowns, the
index \emph{direction} and the link \emph{shape}, and estimates both
nonparametrically. It pays a slower rate for the direction and needs bandwidth or
smoothing choices for the link \citep{hardle1993}. Our setting is a favourable
special case. The elliptical geometry fixes the index as the squared Mahalanobis
radius, supplied by the per-class whitening. The radial generator fixes the
\emph{form} of the link, which leaves only finitely many fractional-power
coefficients (Section~\ref{sec:theory}). A known-form link on a known index is a
parametric problem, so a $\sqrt{n}$ rate is available where the general
single-index model attains only a nonparametric one. Ghosh's spline GAM estimates
the link nonparametrically; this paper takes the parametric route that the
elliptical assumption permits.

\subsection{Stochastic polynomials and the captured-fraction functional}
The polynomial maximization method \citep{kunchenko2002} estimates parameters by
projecting a score onto a finite basis of (possibly fractional) powers, with an
asymptotic efficiency governed by a \emph{captured-fraction} functional: the
fraction of Fisher information retained by the projection. The signed-parity,
continuous-$\alpha$ fractional-power family we use is the parametrically-adaptive
transition polynomial \citep{zabolotnii2026patp}. The apparatus has been applied
to non-Gaussian regression and to inertialess filtering with compressive
\emph{logarithmic} generating elements $\sign(\xi)\log^{i-1}(1+|\xi|)$, which need
only a finite second moment \citep{klopotovskyi2026log}. It also has a moment-free
form, through the empirical characteristic function, for infinite-variance laws
\citep{zabolotnii2026cf}. None of these targets discriminant analysis or
elliptical classification. We use the captured-fraction view only to motivate the
finite fractional-power projection; every identity and rate used here is stated
and proved in the present paper. \citet{zabolotnii2026paper4} treat the Gaussian
projection case, where the projection is onto an affine link and reproduces the
quadratic discriminant. We carry the projection to elliptical carriers, where the
link is non-affine and the same fractional-power basis \emph{derives} the radial
link of \eqref{eq:bridge-intro}.

\subsection{Heavy tails and moment-free estimation}
The parametric route to heavy-tailed classification fits a specific heavy-tailed
family per class and plugs in its Bayes rule, most commonly the multivariate
Student-$t$, as in robust $t$-mixture discrimination \citep{peel2000}. This
commits to one generator and estimates its degrees of freedom. Our derived link
keeps the \emph{radial form} of the generator but leaves the fractional-power
coefficients free. It therefore need not select a single parametric law, and it
reduces to QDA as the tails lighten. Empirical-characteristic-function and
codifference estimators \citep{feuerverger1977,feuerverger1990} handle heavy tails
in the univariate setting. The same idea has been brought into the Kunchenko
apparatus as a moment-free stochastic-polynomial estimator valid for
infinite-variance laws \citep{zabolotnii2026cf}. Its multivariate radial-link
analogue lies at the boundary of our theory (Section~\ref{sec:discussion}), where
the radius loses its second moment.

\subsection{Asymptotic and sieve theory}
The asymptotic results we use are standard; the contribution is their application
to the derived radial link. $M$-estimation under possible misspecification
\citep{vandervaart1998} gives the rate and the sandwich; because the per-class
$(\widehat{\bmu},\widehat{\Sig})$ are estimated first, the radii are
\emph{generated regressors} \citep{pagan1984}, whose first-stage contribution to
the second-stage influence function is the two-step correction of
\citet{newey1994}. Treating the fractional-power family as a sieve, its
consistency and the $\sqrt{n}$-normality of smooth functionals follow the sieve
theory of \citet{chen2007}; the reduction from plug-in excess risk to
link-estimation error follows standard classification theory \citep{devroye1996}.

\section{The elliptical radial-link framework}
\label{sec:framework}

\subsection{Setup}
Two classes $c\in\{0,1\}$ with priors $\pi_c$ have elliptical class-conditional
densities $p_c(\bx)=|\Sig_c|^{-1/2}g_c(\Dsq{c}(\bx))$ with $\Sig_c\succ0$ and
a normalized radial density generator $0<g_c<\infty$ on $[0,\infty)$, i.e.\
$\int_{\reals^d}g_c(\|\bm y\|^2)\,\mathrm d\bm y=1$, so that $g_c$ carries its own
normalizing constant. Write $\genr_c=\log g_c$ for the \emph{radial link}.
We treat the common-generator case $g_0=g_1=g$ ($\genr_0=\genr_1=\genr$) as the
main object and note the class-specific extension where it matters. Let
$\sigma(u)=(1+e^{-u})^{-1}$, $\eta(\bx)=\Prob(c=1\mid\bx)$, $\LLR=\log(\eta/(1-\eta))$
the log-likelihood ratio (LLR), and $\Rstar$ the Bayes risk.

\subsection{The radial-link bridge and GAM membership}
\begin{proposition}[Radial-link bridge; GAM membership]
\label{prop:bridge}
For elliptical class-conditionals as in the Setup --- priors $\pi_c>0$, $\Sig_c\succ0$,
and normalized generators $0<g_c<\infty$ that carry their own normalizing
constants, so $p_c=|\Sig_c|^{-1/2}g_c(\Dsq c)$ exactly ---
\begin{equation}
\label{eq:bridge}
\LLR(\bx)=\genr_1(\Done(\bx))-\genr_0(\Dzero(\bx))
+\tfrac12\log\frac{|\Sig_0|}{|\Sig_1|}+\log\frac{\pi_1}{\pi_0}.
\end{equation}
As a function of $\bx$, $\LLR\in\operatorname{span}\{1,\;\genr_0\circ D_0^2,\;
\genr_1\circ D_1^2\}$: the posterior $\eta=\sigma(\LLR)$ is a logistic generalized
additive model in the two squared Mahalanobis radii whose component functions
are the radial links $\genr_c$.
\end{proposition}
\begin{proof}
Take logs of $\pi_1p_1/(\pi_0p_0)$ and substitute $p_c=|\Sig_c|^{-1/2}
e^{\genr_c(\Dsq{c})}$; \eqref{eq:bridge} is then an additive combination of the
constant, $\genr_0\circ D_0^2$ and $\genr_1\circ D_1^2$.
\end{proof}

Proposition~\ref{prop:bridge} is the squared-radius restatement of
\citet[Theorem~1]{ghosh2025}, who state it in the Mahalanobis distances
$\delta_c=\sqrt{\Dsq{c}}$ and use it to motivate \emph{fitting} the link as a
GAM smooth. We record it to fix notation; the new content is the dichotomy of
Proposition~\ref{prop:dichotomy} and the derived-link theory of
Section~\ref{sec:theory}.

\subsection{The identity-link dichotomy}
The Mahalanobis-GAM is QDA precisely when the link is affine, and that happens
precisely for Gaussian carriers.

\begin{proposition}[Identity link $\iff$ affine generator]
\label{prop:dichotomy}
Take a common generator, $\genr_0=\genr_1=\genr$. The identity-link span
$\operatorname{span}\{1,\mathrm{id}\}$ over functions of a radius equals exactly
the affine functions $\{t\mapsto at+b\}$. Moreover:
\textup{(i)} if $\genr$ is affine, then $\LLR\in\operatorname{span}\{1,D_0^2,D_1^2\}$,
the identity-link span in the two radii, so the log-odds are quadratic in $\bx$ and
the Bayes rule is a quadratic discriminant;
\textup{(ii)} an affine $\genr(t)=at+b$ is a normalizable generator if and only if
$a<0$, and then $p_c$ is the Gaussian density $\mathcal N(\bmu_c,\Sig_c/(-2a))$;
fixing the scale so that $\Sig_c$ is the class covariance gives $a=-\tfrac12$;
\textup{(iii)} let $d\ge2$, let $\genr$ be real-analytic and non-affine on
$(0,\infty)$ --- e.g.\ the Student-$t_\nu$ link
$\genr(t)=-\tfrac{\nu+d}{2}\log(1+t/\nu)$ --- and exclude the nested case
($\bmu_0=\bmu_1$ and $\Sig_1=\lambda\Sig_0$ for some $\lambda>0$). Then
$\LLR\notin\operatorname{span}\{1,D_0^2,D_1^2\}$, even up to a $P_{\bx}$-null set:
a non-identity link is genuinely required.
\end{proposition}
\begin{proof}
A function $g$ of one variable lies in $\operatorname{span}\{1,\mathrm{id}\}$ iff
$g=a\,\mathrm{id}+b$, by definition of the span.

(i) With $\genr(t)=at+b$, \eqref{eq:bridge} reads
$\LLR=a\Done-a\Dzero+\text{const}$.

(ii) In whitened coordinates $\int e^{a\|\bm y\|^2+b}\,\mathrm d\bm y<\infty$ iff
$a<0$, and $e^{a\Dsq{c}}$ is then proportional to the
$\mathcal N(\bmu_c,\Sig_c/(-2a))$ density.

(iii) Suppose $\LLR=A+B\Dzero+C\Done$ $P_{\bx}$-a.e. On
$\Omega=\reals^d\setminus\{\bmu_0,\bmu_1\}$ both radii are positive, so both sides are
continuous there ($\genr$ is analytic on $(0,\infty)$, not necessarily at $0$), and
$P_{\bx}$ has an everywhere-positive density ($g_c>0$); hence the identity holds at
every $\bx\in\Omega$, and \eqref{eq:bridge} gives
\begin{equation}
\label{eq:dich-proof}
\genr(v)-Cv-\bigl(\genr(u)+Bu\bigr)=A'\qquad\text{for all }(u,v)\in\mathcal I,
\end{equation}
where $\mathcal I=\{(\Dzero(\bx),\Done(\bx)):\bx\in\Omega\}$ and $A'$ is a constant.
The map $\bx\mapsto(\Dzero,\Done)$ has differential
$2\bigl(\Sig_0^{-1}(\bx-\bmu_0),\,\Sig_1^{-1}(\bx-\bmu_1)\bigr)^{\T}$. If its rank were
below $2$ for every $\bx$, then writing $\bx=\bmu_0+\bm h$ and
$M=\Sig_1^{-1}\Sig_0$, the vectors $\Sig_0^{-1}\bm h$ and
$M\Sig_0^{-1}\bm h+\Sig_1^{-1}(\bmu_0-\bmu_1)$ would be parallel for every $\bm h$.
Replacing $\bm h$ by $s\bm h$ and comparing the coefficients of $s^2$ and $s$, every
vector would be an eigenvector of $M$, so $M=\lambda I$, and every vector of
$\reals^d$ would be parallel to $\Sig_1^{-1}(\bmu_0-\bmu_1)$, which for $d\ge2$
forces $\bmu_0=\bmu_1$: the nested case. Outside it the rank is $2$ at some $\bx_0$,
necessarily in the open set $\Omega$ (at $\bmu_c$ one row vanishes). The map
$\Omega\to(0,\infty)^2$ is then a submersion near $\bx_0$, hence open there, so
$\mathcal I$ contains an open rectangle
$U\times V$. Fixing $u\in U$ in \eqref{eq:dich-proof} shows that $\genr$ is affine
on the interval $V$; being real-analytic on $(0,\infty)$, $\genr$ is affine on all
of $(0,\infty)$, a contradiction.
\end{proof}

The nested case is a real exception, not a gap in the proof: there
$\Done=\lambda^{-1}\Dzero$, the span reduces to $\operatorname{span}\{1,\Dzero\}$, and
the generator $\genr(t)=-t/2+\varepsilon\sin(2\pi\log t/\log\lambda)$ (normalizable
for small $\varepsilon$) is analytic and non-affine on $(0,\infty)$ yet gives
$\genr(\lambda^{-1}t)-\genr(t)$ affine in $t$. For the Student-$t$ link and
$\lambda\ne1$, $\genr(\lambda^{-1}t)-\genr(t)$ is bounded and non-constant, hence not
affine, so the conclusion of (iii) holds in the nested case for that family. In
$d=1$ the image $\mathcal I$ is a curve and (iii) is not claimed.

Proposition~\ref{prop:dichotomy} is the structural reason the Gaussian
unification of \citet{zabolotnii2026paper4} is tautological while the elliptical
case is not: for $d\ge2$, outside the nested case and among analytic generators,
identity-link adequacy is equivalent to Gaussianity, and the radial curvature
$\genr(t)+t/2$ is exactly what the identity-link span cannot represent.

\subsection{Machine-checked formalization}
\label{sec:lean}
Part of the algebraic content of
Propositions~\ref{prop:bridge}--\ref{prop:dichotomy} is formalized in Lean~4,
without unproven placeholders and with only the standard axioms; the theorem
list is in Appendix~\ref{app:lean}. The formalization covers the bridge, GAM
membership, part~(i) of Proposition~\ref{prop:dichotomy} (affine generator
$\Rightarrow$ identity-link span), the univariate identity ``identity-link span $=$
affine functions'', a non-affine Student-$t$ witness, and the expansion of the
multivariate Mahalanobis radius as a quadratic form, through which the bridge
transfers to dimension $d$. It does \emph{not} cover part~(iii), which needs the
rank argument above and is proved on paper only. Nor does it cover the estimation
theory of Section~\ref{sec:theory}, which is analytic and numerical. Existing Lean
developments in statistics target generalization bounds and empirical-process
theory \citep{zhang2026lean}; here the checked objects are the structural
identities behind a classifier's likelihood ratio. The formalization is a
reproducible check of those identities, not a pillar of the contribution.

\section{Identifiability, consistency, and Bayes-optimality of the derived link}
\label{sec:theory}

Proposition~\ref{prop:bridge} is near-definitional. The technical contribution is
that the link it features is identifiable, estimable at the parametric rate, and
yields a Bayes-optimal classifier under the stated sieve conditions. Throughout,
the estimator is the one used in the experiments. We estimate $(\bmu_c,\Sig_c)$
per class, form the squared radii $\widehat D_c^2$, and build radial features
$\psi_j(t)=t^{p_j}$ on $t\ge0$. (The implementation uses the signed-parity form
$\sign(t)|t|^{p_j}$ of the PATP basis, which coincides with $t^{p_j}$ on $t\ge0$.)
The working set of powers is $p\in\{1,0.5,1.5\}$; the leading $p_1=1$ makes the
Gaussian case terminate at one term. We then fit an $\ell_2$-regularized logistic
regression of the label on the concatenated per-class features, that is, a
logistic GAM whose link is approximated by the fractional-power sieve. The three
powers are the finite truncation used in practice. The approximation theory
(Theorem~\ref{thm:approx}) uses a growing set of powers that adds to the affine
term powers tending to $0$. The $\ell_2$ ridge is a fixed numerical stabilizer
taken to vanish in the asymptotics, so the limiting object is the unregularized
logistic $M$-estimator of Theorem~\ref{thm:clt}.

\subsection{Identifiability of the radial link}
\begin{lemma}[Radial profile; identifiability]
\label{lem:ident}
If $\bx\mid c$ has an absolutely continuous elliptical law with density
$|\Sig_c|^{-1/2}g_c(\Dsq{c}(\bx))$, $\Sig_c\succ0$ and $g_c$ a normalized density
generator (Section~\ref{sec:framework}), the
within-class law of the squared radius $T_c:=\Dsq{c}(\bx)$ has density
$f_{T_c}(t)=(\omega_d/2)\,t^{d/2-1}g_c(t)$ on $t>0$, with
$\omega_d=2\pi^{d/2}/\Gamma(d/2)$. Hence, once the elliptical scale is fixed by
$\Sig_c$, the generator $g_c$ --- and the radial link $\genr_c=\log g_c$ --- is
identified from the law of $T_c$ Lebesgue-almost everywhere on $(0,\infty)$ (a
density is determined by its law only up to null sets), with
\begin{equation}
\label{eq:ident}
\genr_c(t)=\log f_{T_c}(t)-(d/2-1)\log t+\log(2/\omega_d)\qquad\text{for a.e.\ }t>0,
\end{equation}
the constant being common to classes of the same dimension.
\end{lemma}
\begin{proof}
Whiten $\bm y=\Sig_c^{-1/2}(\bx-\bmu_c)$ (density $g_c(\|\bm y\|^2)$, the
Jacobian cancelling $|\Sig_c|^{-1/2}$); then $T_c=\|\bm y\|^2$, and in spherical
coordinates $\mathrm d\bm y=\omega_d r^{d-1}\,\mathrm dr$ with $r^2=t$ gives
$r^{d-1}\mathrm dr=\tfrac12 t^{d/2-1}\mathrm dt$, whence $f_{T_c}$. Invert and
take logs.
\end{proof}

Equation~\eqref{eq:ident} makes ``estimate the link'' well-posed and dictates a
basis in powers of $t$: the Gaussian link is linear ($\genr(t)=-t/2$), the
Student-$t_\nu$ link is logarithmic
($\genr(t)=-\tfrac{\nu+d}{2}\log(1+t/\nu)$). The elliptical scale is not
identified on its own, since $(\Sig,g(\cdot))$ and $(c\Sig,c^{d/2}g(c\cdot))$ give
the same density. Conditioning on $\Sig_c$, as the whitening does, resolves this,
and the decision rule uses $\genr$ only up to a constant. Identifiability makes
the plug-in route legitimate: $\genr$ is a functional of the observable radial
law, so a sieve that captures it estimates the link consistently
(Theorems~\ref{thm:clt}--\ref{thm:approx}).

\subsection{$\sqrt{n}$-consistency of the derived-link estimator}
Fix the sieve dimension $\sieve$ and stack features
$\bZ_{\sieve}^{0}(\bx)=(1,\{\psi_j(D_c^2(\bx))\}_{c,j\le\sieve})$ using the true
class radii. The implemented generated-regressor version is
$\widehat{\bZ}_{n,\sieve}(\bx)=(1,\{\psi_j(\widehat D_c^2(\bx))\}_{c,j\le\sieve})$.
The head is the logistic $M$-estimator
$\widehat{\bbeta}_n=\argmin_{\bbeta}n^{-1}\sum_i
\ell(y_i,\bbeta^{\T}\widehat{\bZ}_{n,\sieve}(\bx_i))$ with logistic loss $\ell$;
let $\bbeta_{\sieve}^{\star}=\argmin_{\bbeta}\E\,\ell(y,\bbeta^{\T}
\bZ_{\sieve}^{0})$ be the pseudo-true coefficient and
$\LLR_{\sieve}^{\star}=\bbeta_{\sieve}^{\star\T}\bZ_{\sieve}^{0}$ the
pseudo-true radial-GAM score.

\begin{assumption}\label{ass}
\textup{(A0)} $(\bx_i,y_i)_{i\le n}$ are i.i.d.\ draws from the two-class elliptical
model of Section~\ref{sec:framework}; in particular $P_{\bx}$ has a Lebesgue density.
\textup{(A1)} $\E\|\bZ_{\sieve}^{0}\|^2<\infty$, i.e.\ $\E\,T_c^{2\bar p}<\infty$ for
$\bar p=\max_j p_j$; for Student-$t_\nu$, $T_c/d\sim F(d,\nu)$ gives
$\E\,T_c^{k}<\infty\iff k<\nu/2$, so a power $p$ needs $p<\nu/4$ (full basis
$\bar p=1.5$: $\nu>6$; the sub-basis $p\le1$: $\nu>4$). \textup{(A2)} the
population risk has a unique minimizer $\bbeta_{\sieve}^{\star}$ with nonsingular
Hessian $H_{\sieve}$. \textup{(A3)} the first stage
$\widehat\gamma=(\widehat{\bmu}_c,\widehat{\Sig}_c)_{c=0,1}$ is asymptotically linear:
$\sqrt n(\widehat\gamma-\gamma)=n^{-1/2}\sum_i\psi_\gamma(\bx_i,y_i)+o_p(1)$ with
$\E\psi_\gamma=0$ and $\E\|\psi_\gamma\|^2<\infty$. The per-class sample means and
covariances satisfy this with
$\psi_\gamma=\pi_c^{-1}\indic\{y=c\}\bigl(\bx-\bmu_c,\;(\bx-\bmu_c)(\bx-\bmu_c)^{\T}-\Sig_c\bigr)$
whenever $\E\|\bx\|^4<\infty$ (Student-$t$: $\nu>4$); $\sqrt n$-consistency alone
would not suffice, since a first stage with a non-Gaussian $\sqrt n$-limit passes
that limit on to $\widehat{\bbeta}_n$. \textup{(A4)} for
every fractional power in the finite sieve, the local expansion of
$\psi_j(\widehat D_c^2)$ is valid in $L^2$. Since
$\|\partial_{\bmu_c}\Dsq c\|=O(T_c^{1/2})$ and $\|\partial_{\Sig_c}\Dsq c\|=O(T_c)$, the
derivative $\partial_\gamma\psi_j(\Dsq c)=p_jT_c^{p_j-1}\partial_\gamma\Dsq c$ is
$O(T_c^{p_j-1/2})$ near $T_c=0$, square-integrable against
$f_{T_c}\propto t^{d/2-1}$ iff $p_j>\tfrac12-\tfrac d4$; for $d\ge2$ this holds
for every $p_j>0$.
\end{assumption}

\begin{theorem}[Rate and asymptotic normality]
\label{thm:clt}
Under Assumption~\ref{ass}, at fixed $\sieve$,
\[
\sqrt{n}\,(\widehat{\bbeta}_n-\bbeta_{\sieve}^{\star})\todist
\mathcal N\!\big(0,\,H_{\sieve}^{-1}\,\Omega_{\sieve}\,H_{\sieve}^{-1}\big),
\]
where $\Omega_{\sieve}=\Var(\xi)$ is the variance of the \emph{full} influence
function $\xi=s(\bx,y;\bbeta_{\sieve}^{\star},\gamma)+G_\gamma\,\psi_\gamma(\bx,y)$,
with $s(\bx,y;\bbeta,\gamma)=\bigl(\sigma(\bbeta^{\T}\bZ_{\sieve}(\bx;\gamma))-y\bigr)
\bZ_{\sieve}(\bx;\gamma)$ the logistic score, $\bZ_{\sieve}(\cdot;\gamma)$ the features
built with first-stage parameters $\gamma$, and
\begin{equation}
\label{eq:Ggamma}
G_\gamma=\partial_\gamma\E\,s(\bx,y;\bbeta_{\sieve}^{\star},\gamma)
=\E\bigl[\sigma'(\bbeta_{\sieve}^{\star\T}\bZ_{\sieve}^{0})\,\bZ_{\sieve}^{0}\,
\bbeta_{\sieve}^{\star\T}\partial_\gamma\bZ_{\sieve}^{0}\bigr]
+\E\bigl[(\sigma(\bbeta_{\sieve}^{\star\T}\bZ_{\sieve}^{0})-y)\,
\partial_\gamma\bZ_{\sieve}^{0}\bigr].
\end{equation}
Moreover $\LLRhat_{\sieve}=\widehat{\bbeta}_n^{\T}\widehat{\bZ}_{n,\sieve}$ is
$\sqrt{n}$-consistent for $\LLR_{\sieve}^{\star}$ uniformly on compacts when every
$p_j\ge\tfrac12$ (in particular for the working set $\{1,0.5,1.5\}$), and uniformly on
compacts bounded away from $\{\bmu_0,\bmu_1\}$ otherwise.
\end{theorem}
\begin{proof}
$\widehat{\bbeta}_n$ solves the two-step estimating equation
$n^{-1}\sum_i s(\bx_i,y_i;\bbeta,\widehat\gamma)=0$. Consistency follows from
convexity of the logistic loss in $\bbeta$, (A2) and uniform convergence of the
empirical risk \citep[Thm.~5.7]{vandervaart1998}. The map
$\gamma\mapsto\E\,s(\bx,y;\bbeta_{\sieve}^{\star},\gamma)$ is differentiable with
Jacobian \eqref{eq:Ggamma}: differentiation under the expectation is licensed by
(A4) near $\Dsq{c}=0$, where
$\partial_\gamma\psi_j(\Dsq c)=p_j(\Dsq c)^{p_j-1}\partial_\gamma\Dsq c$ is singular,
and by (A1) elsewhere; the singular set $\{\bx=\bmu_c\}$ is $P_{\bx}$-null. With
(A0)--(A3) these are the conditions of the two-step $M$-estimation results of
\citet[Sections~6--7]{newey1994}, which give
$\sqrt n(\widehat{\bbeta}_n-\bbeta_{\sieve}^{\star})
=-H_{\sieve}^{-1}n^{-1/2}\sum_i\bigl[s(\bx_i,y_i;\bbeta_{\sieve}^{\star},\gamma)
+G_\gamma\psi_\gamma(\bx_i,y_i)\bigr]+o_p(1)$ and hence the stated limit.

Both terms of \eqref{eq:Ggamma} are needed. Correct specification
($\E[y\mid\bx]=\sigma(\bbeta_{\sieve}^{\star\T}\bZ_{\sieve}^{0})$, e.g.\ Gaussian classes
with $\sieve=1$) removes only the second, residual-weighted term; the first does not
vanish, so the first-stage correction is present even then and the naive sandwich
$H_{\sieve}^{-1}\Var(s)H_{\sieve}^{-1}$ is not the asymptotic covariance. (In a
Gaussian two-class example with $d=2$, $\sieve=1$, $n=4000$ and $2000$ replications,
the naive sandwich overstates the variances of the two radial coefficients
$4.5$--$7.7$-fold, while $H_{\sieve}^{-1}\Omega_{\sieve}H_{\sieve}^{-1}$ matches the
Monte Carlo variance within $3\%$; Appendix~\ref{app:repro}.)

For the score, write
$\LLRhat_{\sieve}-\LLR_{\sieve}^{\star}=(\widehat{\bbeta}_n-\bbeta_{\sieve}^{\star})^{\T}
\widehat{\bZ}_{n,\sieve}+\bbeta_{\sieve}^{\star\T}(\widehat{\bZ}_{n,\sieve}-\bZ_{\sieve}^{0})$.
The first term is $O_p(n^{-1/2})$ uniformly on compacts, where $\widehat{\bZ}_{n,\sieve}$
is bounded. For the second, $\bigl|\widehat D_c-D_c\bigr|\le
\|\widehat\Sig_c^{-1/2}(\bx-\widehat{\bmu}_c)-\Sig_c^{-1/2}(\bx-\bmu_c)\|=O_p(n^{-1/2})$
uniformly on compacts by (A3) \citep[cf.\ the plug-in Mahalanobis-distance
convergence of][Lemma~1]{ghosh2025}, and $\psi_j(\Dsq c)=D_c^{2p_j}$ is locally
Lipschitz in $D_c$ when $2p_j\ge1$; when $2p_j<1$ it is locally Lipschitz on
$\{D_c\ge\epsilon\}$ for every $\epsilon>0$, i.e.\ away from $\bmu_c$.
\end{proof}

The theorem is asymptotic, and under heavy tails it is slow to take hold. The
score carries $T_c^{\bar p}$, so $\Omega_{\sieve}$ is finite under (A1), but
estimating it by sample moments requires $\E\,T_c^{4\bar p}<\infty$, which fails for
the full basis at moderate $\nu$ (Student-$t_{10}$: $\E\,T_c^k<\infty$ iff $k<5$). In
that regime we observed, in a $t_{10}$ version of the example above, that the plug-in
sandwich is unstable across population samples and that the Monte Carlo variance
stays well below it up to $n=64{,}000$, with or without the first stage. We
therefore do not recommend sandwich-based intervals for the $p=1.5$ term under
heavy tails; the bootstrap or the $p\le1$ sub-basis is the safer route.

The \emph{Gaussian-termination} remark: with $p_1=1$ first, $\sieve=1$ gives an
affine-in-radii score, asymptotically equivalent to QDA; higher powers enlarge
the model only when $\genr$ is non-affine, at no efficiency cost on Gaussian data
beyond estimating coefficients that converge to $0$.

\subsection{Approximation and asymptotic Bayes-optimality}
Write $L(f)=\E\,\ell(y,f(\bx))$ for the population logistic risk of a score $f$,
so that $L(\LLR)=\E\,H(\eta(\bx))\le\log2$ with $H$ the binary entropy.

\begin{theorem}[Sieve approximation]
\label{thm:approx}
Let $\mathcal P\subset(0,\infty)$ be the set of sieve powers, accumulating at $0$,
and $\mathcal S_1\subset\mathcal S_2\subset\cdots$ finite-dimensional spans of
$\{1\}\cup\{(\Dsq c)^{p}:c=0,1,\,p\in\mathcal P\}$ whose union $\mathcal S$ contains
all these functions. For $c=0,1$ let $\rho_c$ be the law of $\Dsq c(\bx)$ when
$\bx\sim P_{\bx}$ --- the class mixture, not class $c$ alone --- and let
$\mathcal H_c$ be the $L^2(\rho_c)$-closure of $\operatorname{span}\{1,t^p:p\in\mathcal P\}$.
If $\genr_c\in\mathcal H_c$ for $c=0,1$, then
\textup{(a)} $\LLR$ lies in the $L^2(P_{\bx})$-closure of $\mathcal S$;
\textup{(b)} the pseudo-true sieve scores $\LLR_{\sieve}^{\star}=\argmin_{f\in\mathcal S_{\sieve}}L(f)$
(assumed attained, (A2)) satisfy
\[
0\le L(\LLR_{\sieve}^{\star})-L(\LLR)\le\tfrac18\inf_{f\in\mathcal S_{\sieve}}
\|f-\LLR\|_{L^2(P_{\bx})}^2\;\longrightarrow\;0\qquad(\sieve\to\infty);
\]
\textup{(c)} $\E|\sigma(\LLR_{\sieve}^{\star})-\eta|\le
\bigl\{\tfrac12[L(\LLR_{\sieve}^{\star})-L(\LLR)]\bigr\}^{1/2}\to0$, and
$R(\sign\LLR_{\sieve}^{\star})-\Rstar\le\bigl\{2[L(\LLR_{\sieve}^{\star})-L(\LLR)]\bigr\}^{1/2}\to0$.
\end{theorem}
\begin{proof}
(a) Given $\varepsilon>0$ choose $h_c\in\operatorname{span}\{1,t^p:p\in\mathcal P\}$
with $\|\genr_c-h_c\|_{L^2(\rho_c)}<\varepsilon/2$. By change of variables
$\|\genr_c\circ\Dsq c-h_c\circ\Dsq c\|_{L^2(P_{\bx})}=\|\genr_c-h_c\|_{L^2(\rho_c)}$, and
$h_1\circ\Done-h_0\circ\Dzero+\text{const}\in\mathcal S$; by \eqref{eq:bridge} it is within
$\varepsilon$ of $\LLR$ in $L^2(P_{\bx})$.
(b) For fixed $\bx$ the conditional risk
$\ell_{\bx}(u)=-\eta\log\sigma(u)-(1-\eta)\log(1-\sigma(u))$ is minimized at $u=\LLR(\bx)$
and $\ell_{\bx}''=\sigma'\le\tfrac14$, so $\ell_{\bx}(u)-\ell_{\bx}(\LLR)\le(u-\LLR)^2/8$;
integrate and use $L(\LLR_{\sieve}^{\star})\le L(f)$ for every $f\in\mathcal S_{\sieve}$. The
infimum tends to $0$ by (a), because the $\mathcal S_{\sieve}$ increase to $\mathcal S$.
(c) $\ell_{\bx}(u)-\ell_{\bx}(\LLR)$ is the Kullback--Leibler divergence between
Bernoulli$(\eta)$ and Bernoulli$(\sigma(u))$, at least $2(\eta-\sigma(u))^2$ by
Pinsker's inequality; integrate and apply Cauchy--Schwarz. The risk bound follows
from $R(\sign f)-\Rstar\le2\E|\sigma(f)-\eta|$ \citep[Thm.~2.2]{devroye1996}.
\end{proof}

The hypothesis $\genr_c\in\mathcal H_c$ is on the mixture, not on each class, and it
is not automatic: for a Gaussian class against a Cauchy class,
$\genr_0\circ\Dzero=-\Dzero/2$ is not square-integrable under the Cauchy component,
$\LLR\notin L^2(P_{\bx})$, and no sieve approximates it in $L^2(P_{\bx})$. In the
common-generator Student-$t$ setting of this paper $\genr$ grows logarithmically,
so $\genr\circ\Dsq c\in L^2(P_{\bx})$ whenever both classes are Student-$t$ with
$\nu>0$.

We do \emph{not} claim universal $L^2(\rho_c)$-completeness of the powers: bounded
powers cannot approximate faster-growing functions, and for heavy-tailed radial
laws the moment problem is indeterminate (the log-normal being the textbook
moment-indeterminate case). Theorem~\ref{thm:approx} is therefore a conditional
sieve-approximation statement. The identity $(t^p-1)/p\to\log t$ certifies the
pure log target under the usual dominated-convergence envelope; Student-$t$ links
are treated as log-type targets whose finite-sieve adequacy is checked
numerically in Section~\ref{sec:exp}. Adjoining the element $\log(1+t/\nu)$ itself,
in the manner of the logarithmic generating elements of
\citet{klopotovskyi2026log}, would contain the Student-$t_\nu$ link exactly when
$\nu$ is known; we keep the power sieve because $\nu$ is not, and because powers of
slowly varying elements are strongly collinear and need regularization. This is the elliptical, target-restricted
instance of the captured-fraction principle ($\kappa_{\sieve}\to1$ for this $\LLR$)
\citep{zabolotnii2026paper4}.

\begin{theorem}[Asymptotic Bayes-optimality, iterated limit]
\label{thm:bayes}
Let $\widehat g_{n,\sieve}=\sign(\LLRhat_{\sieve})$, with the risk $R$ computed for a
new observation independent of the sample. Assume (A0)--(A4) at every fixed
$\sieve$ and the hypotheses of Theorem~\ref{thm:approx}. Then
\begin{equation}
\label{eq:risk}
0\le R(\widehat g_{n,\sieve})-\Rstar\le
2\,\E_{\bx}\bigl|\sigma(\LLRhat_{\sieve})-\sigma(\LLR_{\sieve}^{\star})\bigr|
+\bigl\{2[L(\LLR_{\sieve}^{\star})-L(\LLR)]\bigr\}^{1/2},
\end{equation}
the first term tending to $0$ in probability as $n\to\infty$ at fixed $\sieve$ and
the second tending to $0$ as $\sieve\to\infty$. Hence, in the iterated limit,
$\lim_{\sieve\to\infty}\limsup_{n\to\infty}\Prob\{R(\widehat g_{n,\sieve})-\Rstar>\varepsilon\}=0$
for every $\varepsilon>0$: the derived-link plug-in is asymptotically Bayes-optimal
whenever the radial links lie in the closures $\mathcal H_c$ of
Theorem~\ref{thm:approx}.
\end{theorem}
\begin{proof}
By \citet[Thm.~2.2]{devroye1996}, $R(\sign f)-\Rstar\le2\E|\sigma(f)-\eta|$; the
triangle inequality splits this into
$2\E|\sigma(\LLRhat_{\sieve})-\sigma(\LLR_{\sieve}^{\star})|$ and
$2\E|\sigma(\LLR_{\sieve}^{\star})-\eta|$, and Theorem~\ref{thm:approx}(c) bounds the
second. For the first, $|\sigma(a)-\sigma(b)|\le\min\{1,|a-b|/4\}$, so for any compact
$K$ that excludes small open balls around $\bmu_0,\bmu_1$,
$\E_{\bx}|\sigma(\LLRhat_{\sieve})-\sigma(\LLR_{\sieve}^{\star})|\le
\tfrac14\sup_K|\LLRhat_{\sieve}-\LLR_{\sieve}^{\star}|+P_{\bx}(K^{c})$. The supremum is
$O_p(n^{-1/2})$ by Theorem~\ref{thm:clt}, and, because $P_{\bx}$ has a Lebesgue
density (A0) and so no atoms at $\bmu_0,\bmu_1$, $P_{\bx}(K^{c})$ can be made
arbitrarily small by the choice of $K$. No integrable envelope and no $L^1$ rate is
needed.
\end{proof}

The result is the \emph{iterated} limit, which is what the two theorems deliver;
the simultaneous regime $\sieve=\sieve_n\to\infty$ would need a uniform-in-$\sieve$
sieve rate with eigenvalue control on $H_{\sieve_n}$ against power-collinearity,
which we leave open. Finite-$\sieve$ behaviour is quantified in
Section~\ref{sec:exp}.

\subsection{When the derived link beats QDA}
Assume $\E\|\bx\|^2<\infty$. An elliptical class has mean $\bmu_c$ and covariance
$\kappa_c\Sig_c$ with $\kappa_c=\E\,T_c/d$, so the population QDA score --- the
log-ratio of the moment-matched Gaussians, which plug-in QDA estimates --- is
\[
q^{\star}(\bx)=-\frac{\Done(\bx)}{2\kappa_1}+\frac{\Dzero(\bx)}{2\kappa_0}
-\frac12\log\frac{|\kappa_1\Sig_1|}{|\kappa_0\Sig_0|}+\log\frac{\pi_1}{\pi_0}
\;\in\;\operatorname{span}\{1,\Dzero,\Done\}.
\]

\begin{corollary}[When QDA has a risk floor]
\label{cor:dichotomy}
Let the classes be as in Section~\ref{sec:framework} (so $P_{\bx}$ has a Lebesgue
density), and assume $\E\|\bx\|^2<\infty$, that $q^{\star}$ is not constant and that
$P_{\bx}\{\LLR=0\}=0$. Let $\widehat q_n$ be plug-in QDA built from consistent
estimates of the class proportions, means and covariances.
\textup{(a)} $R(\sign\widehat q_n)\to R(\sign q^{\star})$ in probability, and the
limit has excess risk
$R(\sign q^{\star})-\Rstar=\E\bigl[|2\eta-1|\,\indic\{\sign q^{\star}\ne\sign\LLR\}\bigr]$.
\textup{(b)} This floor is positive if and only if
$P_{\bx}\{\sign q^{\star}\ne\sign\LLR\}>0$. Under Gaussianity $q^{\star}=\LLR$ and the
floor is $0$. A non-affine generator does not by itself create a floor: with a
common, strictly decreasing generator, $\Sig_0=\Sig_1$ and $\pi_0=\pi_1$, both
boundaries equal $\{\Dzero=\Done\}$ and the floor is $0$ for every such generator.
\textup{(c)} Whenever the floor is positive and the hypotheses of
Theorem~\ref{thm:bayes} hold, the derived link's excess risk falls below it in
the iterated limit, where it tends to $0$.
In particular a positive score gap $\|q^{\star}-\LLR\|>0$, or a positive gap of the
$L^2(P_{\bx})$ projection of $\LLR$ on $\operatorname{span}\{1,\Dzero,\Done\}$, does
not imply a floor; excess risk depends only on where the signs disagree.
\end{corollary}
\begin{proof}
(a) The plug-in parameters converge in probability, so $\widehat q_n\to q^{\star}$
uniformly on compacts. As a non-constant polynomial of degree at most $2$ in $\bx$,
$q^{\star}$ has a Lebesgue-null zero set, hence $P_{\bx}$-null because $P_{\bx}$ has a
Lebesgue density. For a compact $K$ and
$\delta>0$,
$P_{\bx}\{\sign\widehat q_n\ne\sign q^{\star}\}\le P_{\bx}(K^{c})
+P_{\bx}(K\cap\{|q^{\star}|\le\delta\})+\indic\{\sup_K|\widehat q_n-q^{\star}|>\delta\}$,
and each term can be made small; this bounds $|R(\sign\widehat q_n)-R(\sign q^{\star})|$.
The excess-risk identity is \citet[Thm.~2.2]{devroye1996}.
(b) $|2\eta-1|>0$ off $\{\LLR=0\}$, which is null. Under Gaussianity $\kappa_c=1$ and
$q^{\star}$ coincides with \eqref{eq:bridge}. With a common generator and
$\Sig_0=\Sig_1$, $\kappa_0=\kappa_1$ and the determinants cancel, so with equal
priors $\sign q^{\star}=\sign(\Dzero-\Done)=\sign\LLR$, the last step because $\genr$
is strictly decreasing.
(c) is Theorem~\ref{thm:bayes}.
\end{proof}

Both regimes occur. In one dimension with Student-$t_8$ classes, equal scales and
means $\pm1$, the projection gap is $0.57$ but QDA's excess risk is exactly $0$, and
plug-in QDA's excess risk falls from $1.6\times10^{-3}$ at $n=200$ to
$1.7\times10^{-5}$ at $n=20{,}000$; with means $0,1$ and scales $1,3$ the floor is
$2.4\times10^{-3}$, and plug-in QDA settles at it (Appendix~\ref{app:repro}).

The \emph{magnitude} of the QDA floor depends on how far apart the two boundaries
lie and on $|2\eta-1|$ between them; heuristically it is controlled by two
structural levers.
\emph{(i) Covariance heterogeneity}: with equal class covariances, equal priors and
a common decreasing generator the two boundaries coincide and the floor is exactly
$0$ (Corollary~\ref{cor:dichotomy}(b)); covariance contrast moves them apart and
opens the floor.
\emph{(ii) Tail-heaviness}: for Student-$t$, $\genr''(t)=\tfrac{\nu+d}{2}\nu^{-2}
(1+t/\nu)^{-2}$ rises near $t=0$ as $\nu$ decreases but is non-monotone at large
$t$. Monotonicity of the \emph{integrated} gap is therefore not proven; it is an
empirical regularity (Section~\ref{sec:exp}). We state ``the derived-link
advantage grows with covariance heterogeneity and with tail-heaviness, and
vanishes toward the Gaussian limit'' as a heuristic consistent with, but not
implied by, the curvature picture. In the tested regimes the method matches QDA
where QDA is right and improves on it where the generator departs from the
Gaussian.

\section{Experiments}
\label{sec:exp}

The analysis plans were fixed before the runs. We call them pre-specified rather
than pre-registered, since there is no public registry. Section~\ref{sec:exp-id}
shows that the identity link fails and the fractional basis captures the link;
\S\ref{sec:exp-ghosh} compares the derived link with the fitted Ghosh
Mahalanobis-GAM on real data; \S\ref{sec:exp-heavy} turns to heavy-tailed real data
and the adaptivity curve; \S\ref{sec:exp-rate} is a rate simulation against a
closed-form ground truth. All randomized runs use seed $2026$ and a paired
bootstrap with $R=2000$; Appendix~\ref{app:repro} lists the script behind each
result.

\subsection{Setup}
\label{sec:exp-setup}
\emph{Heads.} \texttt{qda}: textbook quadratic discriminant. \texttt{identity}:
logistic regression linear in the global Mahalanobis radii (identity link).
\texttt{kunchenko}: logistic regression on the fractional/PATP radial basis of
the global radii (the derived link). \texttt{ghosh}: the global
Mahalanobis-distance GAM of \citet{ghosh2025}, in which the per-class
\emph{unsquared} distances $\delta_c=\sqrt{\Dsq{c}}$ enter a \emph{penalized-spline}
GAM with a logistic link and REML-chosen smoothing parameters (following Wood,
2017). We re-implement it with the \texttt{mgcv} package. For the three-class
\texttt{wine} set we fit a multinomial spline GAM with fixed-df smooths, because
the penalized-REML multinomial fit is numerically unstable on these
well-separated classes. This is the elliptic-class classifier of Ghosh et al.\
(their Theorem~1). Their \emph{local} kernel-Mahalanobis extension, with a
bootstrap-selected bandwidth, targets non-elliptic and multimodal classes; it lies
outside the unimodal elliptical scope of this comparison, and a faithful
reproduction is left to future work. The \texttt{kunchenko}/\texttt{ghosh}
comparison is at \emph{equal input budget}: both use the same per-class
Mahalanobis radii. The \texttt{kunchenko} head expands the radii in the
closed-form basis $p\in\{1,0.5,1.5\}$ and fits only linear coefficients; the
\texttt{ghosh} head fits a penalized spline of each radius with data-chosen
smoothness. Only the way the link is obtained differs: derived versus fitted.
\emph{Protocol.} Real-data benchmarks use $5\times5$ repeated stratified
cross-validation with per-fold fitting and standardization (leakage-safe);
estimators are compared by paired bootstrap confidence intervals (CIs) of the
accuracy difference.

\subsection{The identity link fails for elliptical carriers}
\label{sec:exp-id}
We first confirm the structural prediction of Proposition~\ref{prop:dichotomy}
on synthetic elliptical Student-$t$ carriers. Table~\ref{tab:identity} reports
the coefficient of determination $R^2_{\mathrm{lin}}$ of the best affine
(identity-link) fit to the true radial log-density, and the captured fraction
$\kappa_{\sieve}$ of the fractional basis. Identity-link adequacy collapses from
$1.000$ (Gaussian) to $0.240$ at $t_4$; the fractional basis recovers it,
$\kappa_2(t_4)=0.95$, while the Gaussian case terminates at $\sieve=1$ (a single
linear term), reproducing the affine-link collapse of
Proposition~\ref{prop:dichotomy}.

\begin{table}[t]
\centering
\caption{Identity-link adequacy $R^2_{\mathrm{lin}}$ by generator, and the
fractional-basis captured fraction $\kappa_{\sieve}$. The
identity link is exact only for the Gaussian; the fractional basis captures the
non-Gaussian link at low order.}
\label{tab:identity}
\small
\begin{tabular}{@{}lccccc@{}}
\toprule
generator & Gaussian & $t_{30}$ & $t_8$ & $t_4$ & $t_{2.5}$ \\
\midrule
$R^2_{\mathrm{lin}}$ (identity link) & $1.000$ & $0.95$ & $0.66$ & $0.240$ & $0.009$ \\
$\kappa_2$ (fractional basis) & $1.000$ (term.\ $\sieve{=}1$) & --- & --- & $0.95$ & --- \\
\bottomrule
\end{tabular}
\end{table}

\subsection{Derived link versus the fitted Ghosh GAM on real data}
\label{sec:exp-ghosh}
Table~\ref{tab:realdata} compares the heads on two standard, approximately
elliptical multivariate benchmarks. At equal input budget the \emph{derived} link
ties with the \emph{fitted} penalized-spline GAM on both \texttt{breast\_cancer}
($[-0.005,+0.003]$; means $0.959$ vs $0.960$) and \texttt{wine}
($[-0.008,+0.019]$; $0.975$ vs $0.970$). We claim \emph{parity}, not dominance: a
link written down in closed form, with no smoothing-parameter or bandwidth
selection, is statistically indistinguishable from a GAM with data-tuned
smoothness. The derived link also beats the identity link on both sets (CIs
exclude $0$). Against QDA the result depends on the regime. QDA wins on the
light-tailed, near-Gaussian \texttt{wine} ($[-0.027,-0.008]$ for
\texttt{kunchenko}$-$\texttt{qda}) and ties on \texttt{breast\_cancer}. This is
consistent with Corollary~\ref{cor:dichotomy}: these benchmarks are not the regime
where the link should beat QDA, which \S\ref{sec:exp-heavy} examines. The
non-elliptical ten-class \texttt{digits} set, a tie in earlier proxy runs, is
outside the thesis and omitted from the faithful comparison.

\begin{table}[t]
\centering
\caption{Mean accuracy on real benchmarks ($5\times5$ CV) and paired bootstrap
$95\%$ CIs of accuracy differences. \texttt{ghosh} is the faithful
global Mahalanobis-GAM of \citet{ghosh2025} (per-class unsquared distances, mgcv
penalized spline, REML; a fixed-df multinomial spline GAM for the three-class
\texttt{wine}). At equal input budget the
\emph{derived} link (\texttt{kunchenko}) is not significantly worse than this
\emph{fitted} GAM while requiring no
smoothing-parameter selection.}
\label{tab:realdata}
\small
\begin{tabular}{@{}lcccc@{}}
\toprule
dataset $(n,d,C)$ & \texttt{qda} & \texttt{identity} & \textbf{\texttt{kunchenko}} & \texttt{ghosh} (fitted) \\
\midrule
wine $(178,13,3)$ & $\mathbf{0.992}$ & $0.965$ & $0.975$ & $0.970$ \\
breast\_cancer $(569,30,2)$ & $\mathbf{0.964}$ & $0.901$ & $0.959$ & $0.960$ \\
\midrule
\multicolumn{5}{@{}l}{\emph{Paired bootstrap $95\%$ CI of $\Delta$accuracy ($>0$ favours the derived link):}}\\
\texttt{kunchenko}$-$\texttt{ghosh} (fitted) & \multicolumn{2}{l}{wine $[-0.008,+0.019]$} & \multicolumn{2}{l}{bc $[-0.005,+0.003]$}\\
\texttt{kunchenko}$-$\texttt{identity} & \multicolumn{2}{l}{wine $[+0.004,+0.017]$} & \multicolumn{2}{l}{bc $[+0.047,+0.069]$}\\
\texttt{kunchenko}$-$\texttt{qda} & \multicolumn{2}{l}{wine $[-0.027,-0.008]$} & \multicolumn{2}{l}{bc $[-0.011,+0.000]$}\\
\bottomrule
\end{tabular}
\end{table}

\subsection{Heavy-tailed real data and adaptivity}
\label{sec:exp-heavy}
The standard benchmarks are light-tailed. We therefore test on six daily
financial series from three asset classes (FRED): three exchange rates (CAD/USD
\texttt{DEXCAUS}, JPY/USD \texttt{DEXJPUS}, GBP/USD \texttt{DEXUSUK}), crude oil
(\texttt{DCOILWTICO}), an equity index (S\&P~500, \texttt{SP500}) and, as a
light-tailed control, EUR/USD (\texttt{DEXUSEU}). The features are $d=5$
log-return embeddings. The binary label separates the top and bottom terciles of
a $21$-day trailing realized volatility; the middle tercile is dropped. Excess
kurtosis ranges from $2.5$ (EUR/USD) through $6.9$--$9.3$ (the other exchange
rates) to $16.8$ (S\&P~500) and $64.9$ (oil).

Both the embeddings and the label are autocorrelated, so an i.i.d.\ random-CV
bootstrap would understate the confidence intervals. We therefore use a
purged-and-embargoed blocked $6$-fold split (train-only standardization, an
embargo of $26$ samples) and a moving-block bootstrap (block $21$) on the pooled
out-of-sample paired accuracy. For CAD/USD a forward-chaining rolling-origin check
agrees. Table~\ref{tab:fx} reports the result, and two findings hold across the
series. \emph{(i)} At equal budget the derived link is never significantly worse
than the fitted Ghosh MD-GAM. It ties on five series, including the control, and
is better on CAD/USD. The closed-form link thus matches the tuned GAM on exchange
rates, a commodity and an equity index without smoothing selection. \emph{(ii)} The
derived link beats QDA on the three heaviest-tailed series (oil, S\&P~500,
JPY/USD) and ties on the rest, including the light-tailed EUR/USD control. The
advantage follows tail-heaviness, which is the adaptivity claim on real data. An
i.i.d.\ random-CV bootstrap would inflate the QDA edge through temporal leakage;
we report only the leakage-free intervals. The margins between derived and fitted
links are small, as expected for two flexible radial links on the same inputs. The
result against the fitted GAM is parity at lower tuning cost; against QDA it is an
advantage that grows with the tails.

\begin{table}[t]
\centering
\caption{Heavy-tailed real data across asset classes (purged-and-embargoed blocked
$6$-fold, moving-block bootstrap). The fitted comparator is the
faithful global Mahalanobis-GAM of \citet{ghosh2025} (mgcv penalized spline, REML).
Moving-block $95\%$ CI of $\Delta$accuracy ($>0$ favours the derived link);
\textbf{bold} excludes $0$. EUR/USD is a light-tailed control. The derived link is
never significantly worse than the fitted GAM (matching it on all but CAD/USD), and
beats QDA where the tails are heaviest.}
\label{tab:fx}
\small
\begin{tabular}{@{}llrrcc@{}}
\toprule
series & class & $n$ & exc.\ kurt.\ & \texttt{kun}$-$\texttt{ghosh} (fitted) & \texttt{kun}$-$\texttt{qda} \\
\midrule
CAD/USD & FX & $9256$ & $9.3$ & $\mathbf{[+0.004,+0.009]}$ & $[-0.010,+0.017]$ \\
JPY/USD & FX & $9250$ & $9.1$ & $[-0.003,+0.001]$ & $\mathbf{[+0.009,+0.047]}$ \\
GBP/USD & FX & $9254$ & $6.9$ & $[-0.004,+0.002]$ & $[-0.008,+0.022]$ \\
WTI oil & comm.\ & $6772$ & $64.9$ & $[-0.003,+0.007]$ & $\mathbf{[+0.023,+0.069]}$ \\
S\&P 500 & equity & $1658$ & $16.8$ & $[-0.025,+0.007]$ & $\mathbf{[+0.039,+0.125]}$ \\
\midrule
EUR/USD & FX\,(light) & $4574$ & $2.5$ & $[-0.001,+0.003]$ & $[-0.016,+0.013]$ \\
\bottomrule
\end{tabular}
\end{table}

\paragraph{Are the real embeddings elliptical?} The Bayes analysis of
\S\ref{sec:theory} assumes elliptical classes, so we test this on the $d=5$
embeddings the classifiers use, per class and pooled, with three probes (seed
$2026$). \emph{(i)} Mardia's multivariate skewness: elliptical symmetry forces the
population third-order term to zero, so a significant value is direct evidence
against ellipticity. \emph{(ii)} Mardia's multivariate kurtosis, reported for
context only: heavy tails inflate it even for an elliptical $t$ law, so a rejection
is expected. \emph{(iii)} Directional-kurtosis homogeneity: under ellipticity
every whitened one-dimensional projection has the same kurtosis. We compare the
spread over $400$ random directions with a matched multivariate-$t$ null by
parametric bootstrap. Mardia's skewness rejects ellipticity in all $18$ cells
(series $\times$ \{pooled, class~0, class~1\}) at $p<10^{-4}$, so the embeddings are
asymmetric. The directional-kurtosis test rejects nowhere ($p>0.06$). The real
series thus depart from ellipticity through a nonzero \emph{third} moment. The
directional-kurtosis result is no evidence against radial tail homogeneity, but
does not prove it either. We therefore read Table~\ref{tab:fx} as a
\emph{robustness} result rather than as confirmation of an elliptical
data-generating process. The derived link keeps its parity with the fitted GAM and
its edge over QDA under measured departures from elliptical symmetry. Modelling the
asymmetry, for instance with a skew-elliptical or signed radial link, is left to
future work.

To isolate the role of the tails, we keep a \emph{real} covariance structure and
vary only $\nu$. Per-class $(\bmu_c,\Sig_c)$ are estimated from
\texttt{breast\_cancer}, and the classes are drawn as multivariate $t$ with these
parameters and decreasing $\nu$, the covariance held at $\approx\Sig_c$
(Figure~\ref{fig:adapt-a}). The \texttt{kunchenko}$-$\texttt{qda} bootstrap CI is
positive at every $\nu$ and moves upward as the tails grow heavier, from
$[+0.007,+0.012]$ at $\nu=\infty$ (Gaussian) to $[+0.027,+0.041]$ at $\nu=3$; the
trend is not strictly monotone ($[+0.022,+0.032]$ at $\nu=5$ against
$[+0.020,+0.030]$ at $\nu=4$). QDA degrades while the derived link holds. The small advantage at $\nu=\infty$, where
QDA is correctly specified, is a finite-sample effect of the plug-in covariance
estimation shared by both heads, not an approximation gap, and it shrinks with $n$.
The advantage over QDA therefore follows the departure from Gaussianity along a
continuous curve.

\begin{center}
\includegraphics[width=0.9\linewidth]{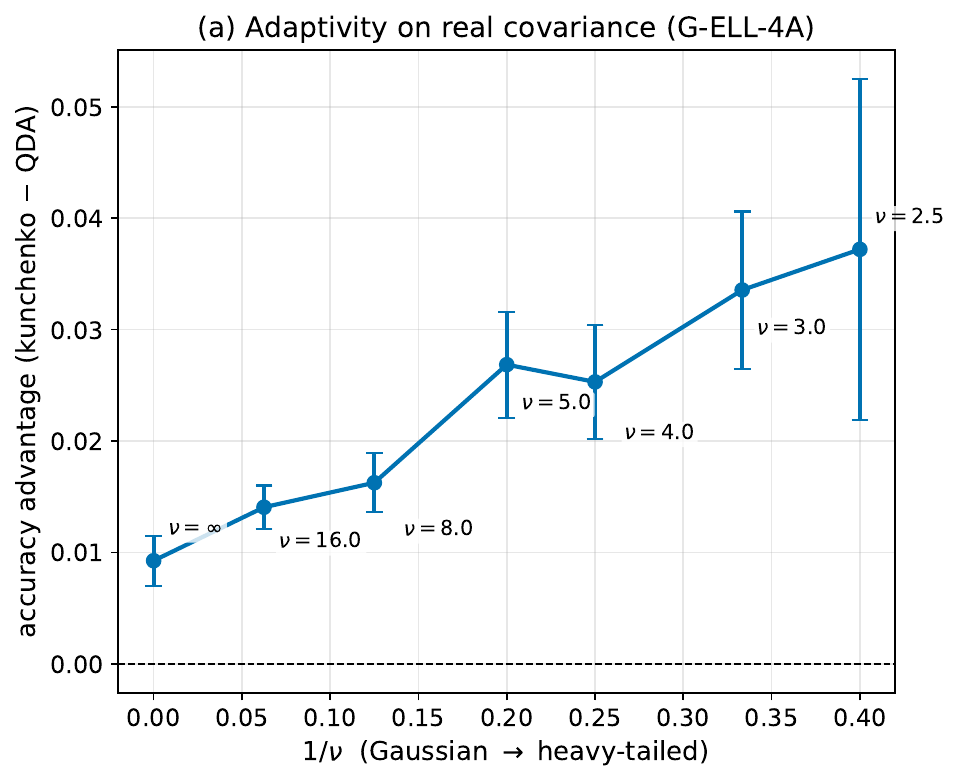}\par
\refstepcounter{figure}\label{fig:adapt-a}
\small\textbf{Figure~\thefigure.} Adaptivity: the \texttt{kunchenko}$-$\texttt{qda}
accuracy advantage on a real covariance structure widens as $\nu$ decreases.
\end{center}

\subsection{Rate simulation against a closed-form ground truth}
\label{sec:exp-rate}
We verify the estimation theory where $\genr$, $\LLR$ and $\Rstar$ are all
closed-form: two-class Student-$t$ with $\bx\mid0\sim t_\nu(\bm0,I)$,
$\bx\mid1\sim t_\nu(\bmu_1,a\,I)$ in $\reals^d$, the scale ratio $a$ being a
covariance-contrast knob. Every regime below sits inside the formal scope
($\nu=8>6$ for the full basis; $\nu>4$ for the plug-in covariance), except the
explicitly illustrative tail $\nu<6$.

\emph{(A) $\sqrt{n}$ rate (Theorem~\ref{thm:clt}).} With true whitening and a
fixed feature scaler, the estimator error $\|\widehat{\bbeta}_n-\bbeta^{\star}\|$
decays with log-log slope $-0.57$ (theory: $-0.5$) over
$n\in[10^3,3.2\!\times\!10^4]$, consistent with the parametric rate.

\emph{(B) Excess-risk dichotomy (Corollary~\ref{cor:dichotomy}).} At $\nu=8$ and
$a=3$, inside the formal scope, QDA's excess risk plateaus at $\approx0.018$:
under the covariance contrast its quadratic boundary departs from the Bayes
boundary of the logarithmic $t$-link. The derived link's excess risk falls from
$0.029$ to $0.0005$ as $n$ grows.

\emph{(B$'$) Sieve closes the gap (Theorem~\ref{thm:approx}).} Holding $n$ fixed
($n=6400$, $\nu=3$, $a=3$) and growing the sieve dimension $\sieve$, the
derived link's excess risk falls monotonically as more fractional powers are
added: $\sieve=1$ (affine) $0.052\to\sieve=2\;0.023\to\sieve=3\;0.018\to
\sieve=4\;0.017\to\sieve=6\;0.014$. The approximation floor shrinks with
$\sieve$, the empirical counterpart of Theorem~\ref{thm:approx}.

\emph{(C) The two levers.} Figure~\ref{fig:adapt-bc} isolates them at fixed $n$.
The QDA gap is $\approx0$ for equal covariance ($a=1$), opens with covariance
contrast (peaking near $a{=}2$--$3$), and grows monotonically as the tails become
heavier (from $0.008$ at $\nu{=}12$ to $0.094$ at $\nu{=}3$, $a{=}3$). The
advantage needs covariance heterogeneity, heavy tails, or both. This is why the
real benchmarks with per-class $\Sig$ show it and an equal-covariance toy example
hides it.

\begin{center}
\includegraphics[width=\textwidth]{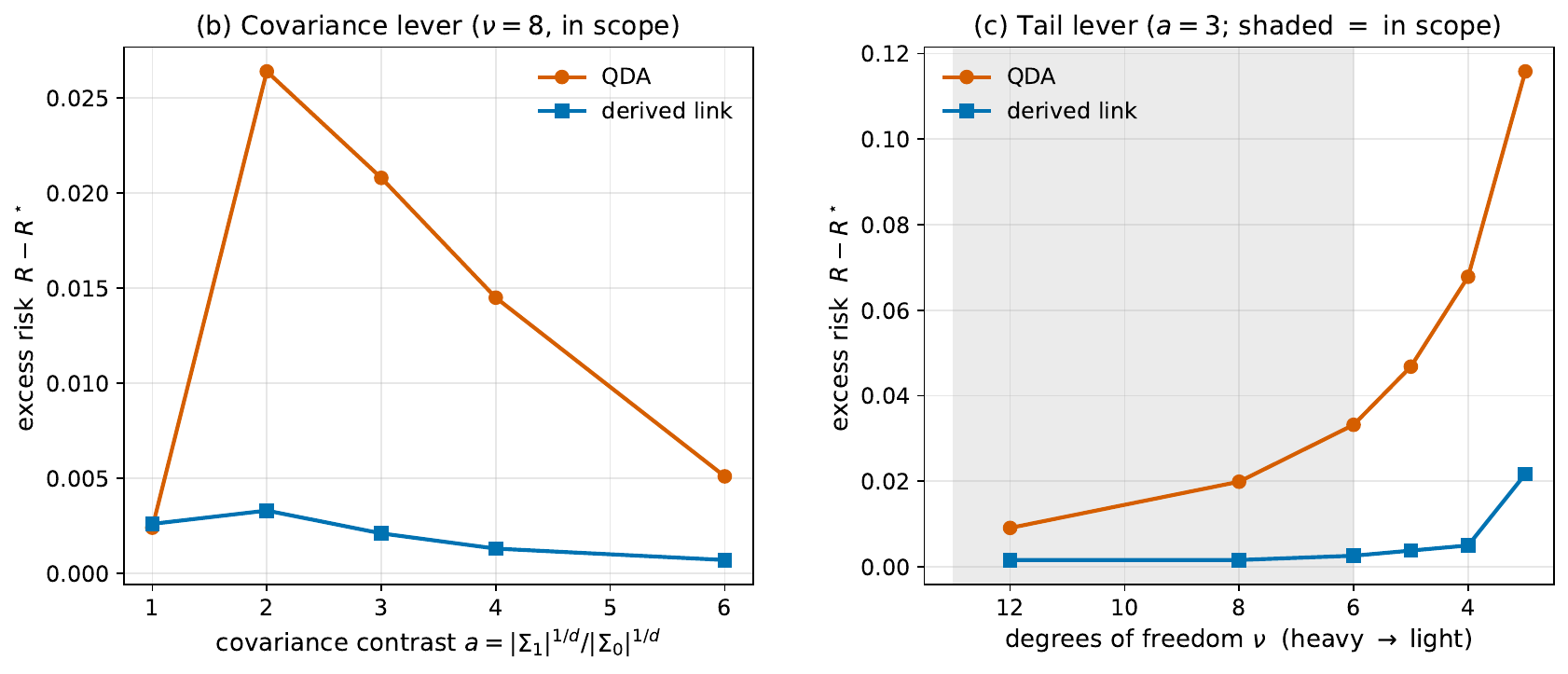}\par
\refstepcounter{figure}\label{fig:adapt-bc}
\small\textbf{Figure~\thefigure.} The two levers of the excess-risk gap against
a closed-form ground truth: QDA excess vs the derived link as a
function of covariance contrast $a$ (at $\nu=8$, panel b) and of tail-heaviness
$\nu$ (at $a=3$, panel c). The gap vanishes at equal covariance and toward the
Gaussian limit.
\end{center}

\section{Discussion and limitations}
\label{sec:discussion}

\paragraph{What the paper establishes.} For elliptical classes the posterior
log-odds are additive in the Mahalanobis radii, with the radial link as component
function (Proposition~\ref{prop:bridge}). For $d\ge2$, non-nested class parameters
and analytic generators, the identity link is adequate exactly under Gaussianity
(Proposition~\ref{prop:dichotomy}). The link is identifiable
(Lemma~\ref{lem:ident}), the fractional-sieve plug-in estimator is
$\sqrt{n}$-consistent (Theorem~\ref{thm:clt}), and the induced classifier is
asymptotically Bayes-optimal in the iterated limit (Theorem~\ref{thm:bayes}).
Empirically, the derived link is never significantly worse than the fitted Ghosh
Mahalanobis-GAM at equal budget, on two UCI benchmarks and six financial series,
and needs no smoothing-parameter or bandwidth selection. It beats QDA on the
heaviest-tailed real series, and the advantage tracks tail-heaviness.

\paragraph{Scope and caveats.}
\emph{(i) Iterated vs simultaneous limit.} Theorem~\ref{thm:bayes} is an iterated
limit. A rate for $\sieve_n\to\infty$ jointly with $n$ needs eigenvalue control on
the Hessian against the near-collinearity of high powers, which we do not prove.
We conjecture that $\sieve_n=o(n^{1/2})$ with a Tikhonov-regularized Hessian
suffices.
\emph{(ii) Moments.} The $\sqrt{n}$ theory needs $\nu>6$ for the full basis and
$\nu>4$ for the sub-basis and the plug-in covariance. Below $\nu=4$ the sample
covariance is no longer $\sqrt{n}$-consistent and the asymptotics do not apply;
the moment-free route through the empirical characteristic function
\citep{feuerverger1990,zabolotnii2026cf} is the appropriate tool there.
\emph{(iii) Equal covariances.} With equal covariances and equal priors, QDA's
floor is exactly zero for every common decreasing generator
(Corollary~\ref{cor:dichotomy}(b), Figure~\ref{fig:adapt-bc}). This explains why a
location-shift toy example hides the advantage while real data with per-class
$\Sig$ show it. It also bounds the claim: the derived link beats QDA only where
the generator's curvature moves the Bayes boundary.
\emph{(iv) High dimension.} Our benchmarks have $d\le64$, and the practical
advantage is a moderate-$d$ one. In a direct probe of $d\gg n$ (ridge-regularized
per-class scatter, $\nu=8$, $a=3$; Appendix~\ref{app:repro}) the advantage does not
survive. As $d/n$ passes one, the radii read off a rank-deficient, ridge-loaded
scatter saturate. The \emph{fixed} fractional-power basis then collapses to the
identity link ($d/n=5$: both at chance, $0.50$), while a \emph{data-adaptive}
quantile-spline link still separates ($0.63$); normalizing the radius by $d$ does
not help. The structural and identifiability results
(Propositions~\ref{prop:bridge}--\ref{prop:dichotomy}, Lemma~\ref{lem:ident}) are
algebraic identities in the radii and hold in any dimension. The practical benefit,
however, rests on well-estimated radii, hence on $n\gg d$. Recovering it at
$d\gg n$ needs a high-dimensional scatter and a data-adaptive or moment-free radius
basis. Spatial-sign direct discriminant rules, which attain optimal rates for
heavy-tailed elliptical classes in high dimension
\citep{zhuang2026sslda,shen2025ssqda}, are a natural source of such a scatter.
\emph{(v) Misspecified generator.} The rate simulation assumes a known parametric
family; robustness to a different data generator is a natural next check. On the
real financial series ellipticity is rejected through asymmetry
(\S\ref{sec:exp-heavy}), so those results are evidence of robustness to mild
non-ellipticity, not confirmations of it.
\emph{(vi) Comparator scope.} We compare against the \emph{global}
Mahalanobis-GAM of \citet{ghosh2025} (their Theorem~1). Their \emph{local}
kernel-Mahalanobis extension, built for multimodal and non-elliptic classes, lies
outside the unimodal elliptical regime studied here. A faithful reproduction of
it, together with the non-elliptical \texttt{digits} benchmark, is left to future
work.

\paragraph{Where the theory proves and where the experiments extend.} The proven
regime and the empirical one do not coincide (Table~\ref{tab:regimemap}). We map
each real series to a tail index through the univariate excess kurtosis,
$\kappa_e=6/(\nu-4)$; the multivariate kurtosis is larger, so these are upper
bounds on $\nu$. The largest real advantages over QDA (oil, S\&P~500, JPY/USD) sit
at $\nu\in(4,4.7)$. This is inside the band of the $p\le1$ sub-basis but below the
$\nu>6$ scope of the full basis, and oil lies essentially at $\nu=4$, where even
$\widehat{\Sig}$ stops being $\sqrt{n}$-consistent. There the $p=1.5$ term is a
regularized finite-sample feature outside the formal CLT. We present these cells
as an \emph{empirical extension} of the proven core, not as covered by the
theorems. The proven statements are the dichotomy (Corollary~\ref{cor:dichotomy})
and the in-scope rate simulation.

\begin{table}[t]
\centering
\caption{Regime map: where the theory \emph{proves} and where the experiments
\emph{extend}. Tail index $\nu$ for the real series is an upper bound from the
univariate excess kurtosis via $\kappa_e=6/(\nu-4)$. Scope thresholds are the moment
conditions of \S\ref{sec:theory} (full basis needs $\nu>6$; the $p\le1$ sub-basis and
the $\sqrt{n}$ plug-in covariance need $\nu>4$).}
\label{tab:regimemap}
\small
\resizebox{\textwidth}{!}{%
\begin{tabular}{@{}llll@{}}
\toprule
Regime ($\nu$) & CLT theory covers & Empirical finding & Real series in band \\
\midrule
Gaussian / $\nu\gtrsim30$ & full basis (terminates at $\sieve{=}1$) & derived ties QDA (Cor.~\ref{cor:dichotomy}) & wine (light UCI) \\
$\nu>6$ & \textbf{full} basis $\{1,\tfrac12,\tfrac32\}$ & derived $\to$ Bayes; QDA gap flat & EUR/USD ($\nu\!\approx\!6.4$) \\
$4<\nu\le6$ & \textbf{sub-basis} $p\le1$ only & derived $\ge$ fitted; beats QDA & GBP, JPY, CAD, S\&P ($\nu\!\approx\!4.4$--$4.9$) \\
$\nu\le4$ & \textbf{none} ($\widehat{\Sig}$ not $\sqrt{n}$) & win vs QDA empirical; CF route proper & WTI oil ($\nu\!\approx\!4.1$, boundary) \\
\bottomrule
\end{tabular}}
\end{table}

\paragraph{Beyond the fourth-moment assumption.} The condition $\nu>4$ is
sufficient, not necessary, and enters only through the plug-in whitening. The
radial-link \emph{structure} (Proposition~\ref{prop:bridge}) and the
identifiability of $\genr$ (Lemma~\ref{lem:ident}) hold for any generator. The
$\sqrt{n}$ theory needs a $\sqrt{n}$-consistent affine-equivariant scatter, and
the sample covariance is the wrong tool once $\nu\le4$. Tyler's distribution-free
$M$-estimator \citep{tyler1987} depends on the data only through the directions
$(\bx-\bmu)/\|\bx-\bmu\|$, so it estimates the shape of $\Sig_c$ at the $\sqrt{n}$
rate for every elliptical law, with no moment condition on the radius (given a
robust location estimate); \citet{ollila2012} survey this and related estimators
in the complex elliptically symmetric setting. It recovers $\Sig_c$ only up to a scalar factor, but for the
power features this factor is absorbed by the coefficients, since
$(ct)^{p}=c^{p}t^{p}$. The Remark to \citet[Lemma~1]{ghosh2025} anticipates this
substitution. The radial features
themselves remain valid for powers $p<\nu/4$, and a moment-free radial link built
on the empirical characteristic function, in the spirit of
\citet{zabolotnii2026cf}, would reach the infinite-variance regime. The derived-link
machinery is thus modular in the scatter plug-in. Heavy-tailed applications such
as the returns of \S\ref{sec:exp-heavy} are accessible by this substitution
without changing the structure or the identifiability argument.

\paragraph{Relation to the Gaussian unification.} The structural layer
(\S\ref{sec:framework}) extends a Gaussian likelihood-ratio-projection
unification \citep{zabolotnii2026paper4} that is a tautology when the carrier is
Gaussian. In the elliptical extension the link is no longer the identity, the
comparison with a fitted GAM becomes meaningful, and the estimation theory of
\S\ref{sec:theory}, absent from the Gaussian case, carries the result.

\section{Conclusion}
\label{sec:conclusion}

The log-likelihood ratio of two elliptical classes is additive in the two
Mahalanobis radii, with the radial link as component function. In two or more
dimensions, and outside the nested case, the Gaussian is the only generator for
which that link is the identity and the rule is quadratic discriminant analysis.
For any other analytic generator the link can be \emph{derived} from the generator
instead of being fitted nonparametrically. We proved that the derived link is
identifiable, $\sqrt{n}$-consistent and asymptotically Bayes-optimal, and part of
the structural algebra is checked in Lean~4. Empirically, the derived link is never
significantly worse than the fitted global Mahalanobis-GAM of Ghosh et al.\ at
equal budget, and needs no smoothing-parameter selection. On six financial series
from three asset classes, under validation that respects temporal dependence, it
beats QDA on the heaviest-tailed series and ties on a light-tailed control. The
method improves on QDA as the generator departs from the Gaussian and matches it
otherwise.

Three extensions follow. A moment-free radial link based on the empirical
characteristic function would carry the method past the $\nu>4$ boundary into the
infinite-variance regime. A rate uniform in the sieve size would turn the
iterated-limit optimality into a statement for $\sieve_n\to\infty$ jointly with
$n$. A high-dimensional ($d\gg n$) treatment with a regularized per-class scatter
would extend the method beyond the moderate dimensions tested here. In each case
the starting point is the same: for elliptical classes the link is not a modelling
choice but a functional of the radius law that can be estimated.

\section*{Reproducibility}
The estimators, the gate scripts, the data-processing and the closed-form rate
simulation, together with the machine-checked Lean~4 development, are available
as a standalone public repository at
\url{https://github.com/SZabolotnii/DSGE-MV_Elliptical_LLR-code-supplement}. It
provides the run scripts (\texttt{code/}), the recorded gate outputs
(\texttt{results/}), the figure-generation script (\texttt{figures/}), the
\texttt{sorry}-free \texttt{EllipticalUnification.lean} (\texttt{lean/}) and the
per-gate verdicts (\texttt{docs/}); \texttt{RUN\_ALL.md} documents the end-to-end
pipeline and \texttt{requirements.txt} the environment. All randomized
experiments use seed $2026$.

\section*{Ethical considerations}
This study uses only public benchmark datasets and public financial time series;
it does not use private, clinical, human-subject, or personally identifiable
data. The financial-series experiments are statistical stress tests of
heavy-tailed regimes, not trading or investment advice. The computations are
lightweight repeated-CV simulations and formal checks; no large-scale model
training or human-subject intervention is involved.

\section*{CRediT authorship contribution statement}
\textbf{Serhii Zabolotnii:} Conceptualization, Methodology, Formal analysis,
Software, Investigation, Validation, Writing -- original draft, Writing -- review
\& editing.

\section*{Declaration of competing interest}
The author declares no competing financial interests or personal relationships
that could have appeared to influence the work reported in this paper.

\section*{Funding}
This research received no specific grant from any funding agency in the public,
commercial, or not-for-profit sectors.

\section*{Data availability}
All data used are public. The benchmark datasets (\texttt{wine} and
\texttt{breast\_cancer}) are distributed with
\texttt{scikit-learn} and originate from the UCI repository; the financial series
are publicly available FRED daily series (\texttt{DEXCAUS}, \texttt{DEXJPUS},
\texttt{DEXUSUK}, \texttt{DEXUSEU}, \texttt{DCOILWTICO}, \texttt{SP500}). All processing scripts, estimators, gates, the closed-form
rate simulation, and the Lean~4 development are released at the repository linked
under Reproducibility.

\section*{Declaration of generative AI and AI-assisted technologies in the manuscript preparation process}
During the preparation of this work the author used AI-assisted coding and
writing environments, including OpenAI Codex and Anthropic Claude, for language
editing, organization of the literature, and scaffolding of the reproducibility
scripts. The author designed the study, derived and verified all theoretical
results (including the \texttt{sorry}-free Lean~4 development), executed and
validated all experiments, and reviewed and edited all content. The author takes
full responsibility for the content of the publication.

\section*{Acknowledgements}
The author thanks the Kunchenko research school for the stochastic-polynomial
apparatus underlying the derived radial link.

\appendix

\section{The Lean development}
\label{app:lean}
Table~\ref{tab:lean} lists the theorems of \texttt{EllipticalUnification.lean}
(Lean~4, Mathlib v4.26.0, \texttt{sorry}-free, axiom footprint
\texttt{[propext, Classical.choice, Quot.sound]}) and the claim of
Section~\ref{sec:framework} that each one restates. Part~(iii) of
Proposition~\ref{prop:dichotomy} is not among them; it is proved on paper.

\begin{center}\begin{minipage}{\linewidth}\centering
\refstepcounter{table}\label{tab:lean}
\small\textbf{Table~\thetable.} Lean theorems and the structural claims they certify.\par\medskip
\begin{tabular}{@{}l>{\raggedright\arraybackslash}p{7.6cm}@{}}
\toprule
Lean theorem & Statement \\
\midrule
\texttt{ellLLR\_eq\_radial\_difference} & the bridge \eqref{eq:bridge} (any generator, any dimension) \\
\texttt{ellLLR\_mem\_span\_radial} & $\LLR\in\operatorname{span}\{1,\genr\circ D_0^2,\genr\circ D_1^2\}$ (GAM membership) \\
\texttt{affine\_link\_mem\_identity\_span} & affine $\genr\Rightarrow\LLR\in\operatorname{span}\{1,D_0^2,D_1^2\}$ (Proposition~\ref{prop:dichotomy}(i)) \\
\texttt{mem\_affine\_span\_iff\_affine} & the identity-link span of one variable is exactly the affine functions \\
\texttt{sq\_not\_mem\_affine\_span} & the radial link $t\mapsto t^2$ is not in that span \\
\texttt{tGen\_not\_affine} & the Student-$t$ generator is non-affine \\
\texttt{mahalaMV\_expand} & the multivariate squared Mahalanobis radius is quadratic $+$ linear $+$ constant \\
\bottomrule
\end{tabular}
\end{minipage}\end{center}

\section{Reproducibility map}
\label{app:repro}
Each result is produced by the scripts of the code supplement (\texttt{code/};
the recorded outputs are in \texttt{results/}, the verdicts in \texttt{docs/}).
Table~\ref{tab:repro} gives the correspondence; all randomized runs use seed
$2026$ unless stated.

\begin{center}\begin{minipage}{\linewidth}\centering
\refstepcounter{table}\label{tab:repro}
\small\textbf{Table~\thetable.} Results and the scripts that produce them.\par\medskip
\begin{tabular}{@{}>{\raggedright\arraybackslash}p{5.4cm}>{\raggedright\arraybackslash}p{6.2cm}@{}}
\toprule
Result & Script(s) \\
\midrule
Table~\ref{tab:identity} (identity-link adequacy, captured fraction) & \texttt{run\_ell\_gate.py} \\
Table~\ref{tab:realdata} (real benchmarks) & \texttt{run\_ell\_gate3.py}\newline \texttt{run\_ell\_gate\_mgcv\_faithful.py}\newline \texttt{run\_ell\_wine\_multinom\_faithful.py} \\
Figure~\ref{fig:adapt-a} ($\nu$-sweep on real covariance) & \texttt{run\_ell\_gate4.py} \\
CAD/USD blocked and forward-chaining validation (\S\ref{sec:exp-heavy}) & \texttt{run\_ell\_gate4b\_blocked.py} \\
Table~\ref{tab:fx} (six financial series) & \texttt{run\_ell\_gate4c\_multiseries.py}\newline \texttt{run\_ell\_gate4c\_mgcv\_faithful.py} \\
Ellipticity test of the financial series (\S\ref{sec:exp-heavy}) & \texttt{run\_ell\_ellipticity\_test.py} \\
Rate simulation and Figure~\ref{fig:adapt-bc} (\S\ref{sec:exp-rate}) & \texttt{run\_ell\_gate5\_rates.py} \\
High-dimensional probe (\S\ref{sec:discussion}, item (iv)) & \texttt{run\_ell\_gate6\_hdlss.py} \\
QDA floor examples after Corollary~\ref{cor:dichotomy} (seed $20261105$) & \texttt{check\_cor\_dichotomy.py} \\
First-stage correction example after Theorem~\ref{thm:clt} (seed $20261105$) & \texttt{check\_thm\_clt.py} \\
Heavy-tailed sandwich remark after Theorem~\ref{thm:clt} & \texttt{check\_thm\_clt\_nsweep.py}\newline \texttt{check\_thm\_clt\_diag.py}\newline \texttt{check\_thm\_clt\_truegamma.py} \\
\bottomrule
\end{tabular}
\end{minipage}\end{center}

\bibliographystyle{cas-model2-names}
\bibliography{references}

@article{ghosh2025,
  author  = {Ghosh, Annesha and Ghosh, Anil K. and SahaRay, Rita and Sarkar, Soham},
  title   = {Classification Using Global and Local {Mahalanobis} Distances},
  journal = {Journal of Multivariate Analysis},
  volume  = {207},
  pages   = {105417},
  year    = {2025},
  doi     = {10.1016/j.jmva.2025.105417}
}

@book{fang1990,
  author    = {Fang, Kai-Tai and Kotz, Samuel and Ng, Kai Wang},
  title     = {Symmetric Multivariate and Related Distributions},
  publisher = {Chapman and Hall},
  address   = {London},
  year      = {1990}
}

@article{cambanis1981,
  author  = {Cambanis, Stamatis and Huang, Steel and Simons, Gordon},
  title   = {On the Theory of Elliptically Contoured Distributions},
  journal = {Journal of Multivariate Analysis},
  volume  = {11},
  number  = {3},
  pages   = {368--385},
  year    = {1981},
  doi     = {10.1016/0047-259X(81)90082-8}
}

@book{hastie1990,
  author    = {Hastie, Trevor J. and Tibshirani, Robert J.},
  title     = {Generalized Additive Models},
  publisher = {Chapman and Hall},
  address   = {London},
  year      = {1990},
  isbn      = {978-0-412-34390-2}
}

@book{mclachlan2004,
  author    = {McLachlan, Geoffrey J.},
  title     = {Discriminant Analysis and Statistical Pattern Recognition},
  publisher = {Wiley},
  address   = {Hoboken, NJ},
  year      = {2004}
}

@book{kunchenko2002,
  author    = {Kunchenko, Yuriy P.},
  title     = {Polynomial Parameter Estimations of Close to {Gaussian} Random Variables},
  publisher = {Shaker Verlag},
  address   = {Aachen},
  year      = {2002}
}

@book{vandervaart1998,
  author    = {van der Vaart, Aad W.},
  title     = {Asymptotic Statistics},
  publisher = {Cambridge University Press},
  address   = {Cambridge},
  year      = {1998}
}

@incollection{newey1994,
  author    = {Newey, Whitney K. and McFadden, Daniel},
  title     = {Large Sample Estimation and Hypothesis Testing},
  booktitle = {Handbook of Econometrics},
  editor    = {Engle, Robert F. and McFadden, Daniel L.},
  volume    = {4},
  pages     = {2111--2245},
  publisher = {Elsevier},
  year      = {1994},
  doi       = {10.1016/S1573-4412(05)80005-4}
}

@book{devroye1996,
  author    = {Devroye, Luc and Gy\"{o}rfi, L\'{a}szl\'{o} and Lugosi, G\'{a}bor},
  title     = {A Probabilistic Theory of Pattern Recognition},
  publisher = {Springer},
  address   = {New York},
  year      = {1996},
  doi       = {10.1007/978-1-4612-0711-5}
}

@article{feuerverger1977,
  author  = {Feuerverger, Andrey and Mureika, Roman A.},
  title   = {The Empirical Characteristic Function and Its Applications},
  journal = {The Annals of Statistics},
  volume  = {5},
  number  = {1},
  pages   = {88--97},
  year    = {1977},
  doi     = {10.1214/aos/1176343742}
}

@article{feuerverger1990,
  author  = {Feuerverger, Andrey},
  title   = {An Efficiency Result for the Empirical Characteristic Function in Stationary Time-Series Models},
  journal = {The Canadian Journal of Statistics},
  volume  = {18},
  number  = {2},
  pages   = {155--161},
  year    = {1990},
  doi     = {10.2307/3315564}
}

@article{tyler1987,
  author  = {Tyler, David E.},
  title   = {A Distribution-Free {M}-Estimator of Multivariate Scatter},
  journal = {The Annals of Statistics},
  volume  = {15},
  number  = {1},
  pages   = {234--251},
  year    = {1987},
  doi     = {10.1214/aos/1176350263}
}

@article{ollila2012,
  author  = {Ollila, Esa and Tyler, David E. and Koivunen, Visa and Poor, H. Vincent},
  title   = {Complex Elliptically Symmetric Distributions: Survey, New Results and Applications},
  journal = {IEEE Transactions on Signal Processing},
  volume  = {60},
  number  = {11},
  pages   = {5597--5625},
  year    = {2012},
  doi     = {10.1109/TSP.2012.2212433}
}

@article{fishbone2024,
  author  = {Fishbone, Justin and Mili, Lamine},
  title   = {New Highly Efficient High-Breakdown Estimator of Multivariate Scatter and Location for Elliptical Distributions},
  journal = {The Canadian Journal of Statistics},
  volume  = {52},
  number  = {2},
  pages   = {437--460},
  year    = {2024},
  doi     = {10.1002/cjs.11770}
}

@article{li2012,
  author  = {Li, Jun and Cuesta-Albertos, Juan A. and Liu, Regina Y.},
  title   = {{DD}-Classifier: Nonparametric Classification Procedure Based on {DD}-Plot},
  journal = {Journal of the American Statistical Association},
  volume  = {107},
  number  = {498},
  pages   = {737--753},
  year    = {2012},
  doi     = {10.1080/01621459.2012.688462}
}

@article{ichimura1993,
  author  = {Ichimura, Hidehiko},
  title   = {Semiparametric Least Squares ({SLS}) and Weighted {SLS} Estimation of Single-Index Models},
  journal = {Journal of Econometrics},
  volume  = {58},
  number  = {1--2},
  pages   = {71--120},
  year    = {1993},
  doi     = {10.1016/0304-4076(93)90114-K}
}

@article{hardle1993,
  author  = {H\"{a}rdle, Wolfgang and Hall, Peter and Ichimura, Hidehiko},
  title   = {Optimal Smoothing in Single-Index Models},
  journal = {The Annals of Statistics},
  volume  = {21},
  number  = {1},
  pages   = {157--178},
  year    = {1993},
  doi     = {10.1214/aos/1176349020}
}

@article{ghosh2005depth,
  author  = {Ghosh, Anil K. and Chaudhuri, Probal},
  title   = {On Maximum Depth and Related Classifiers},
  journal = {Scandinavian Journal of Statistics},
  volume  = {32},
  number  = {2},
  pages   = {327--350},
  year    = {2005},
  doi     = {10.1111/j.1467-9469.2005.00423.x}
}

@article{hubert2017,
  author  = {Hubert, Mia and Rousseeuw, Peter J. and Segaert, Pieter},
  title   = {Multivariate and Functional Classification Using Depth and Distance},
  journal = {Advances in Data Analysis and Classification},
  volume  = {11},
  number  = {3},
  pages   = {445--466},
  year    = {2017},
  doi     = {10.1007/s11634-016-0269-3}
}

@article{peel2000,
  author  = {Peel, David and McLachlan, Geoffrey J.},
  title   = {Robust Mixture Modelling Using the $t$ Distribution},
  journal = {Statistics and Computing},
  volume  = {10},
  number  = {4},
  pages   = {339--348},
  year    = {2000},
  doi     = {10.1023/A:1008981510081}
}

@article{hastie1994,
  author  = {Hastie, Trevor and Tibshirani, Robert and Buja, Andreas},
  title   = {Flexible Discriminant Analysis by Optimal Scoring},
  journal = {Journal of the American Statistical Association},
  volume  = {89},
  number  = {428},
  pages   = {1255--1270},
  year    = {1994},
  doi     = {10.1080/01621459.1994.10476866}
}

@incollection{chen2007,
  author    = {Chen, Xiaohong},
  title     = {Large Sample Sieve Estimation of Semi-Nonparametric Models},
  booktitle = {Handbook of Econometrics},
  editor    = {Heckman, James J. and Leamer, Edward E.},
  volume    = {6B},
  pages     = {5549--5632},
  publisher = {Elsevier},
  year      = {2007},
  doi       = {10.1016/S1573-4412(07)06076-X}
}

@article{pagan1984,
  author  = {Pagan, Adrian},
  title   = {Econometric Issues in the Analysis of Regressions with Generated Regressors},
  journal = {International Economic Review},
  volume  = {25},
  number  = {1},
  pages   = {221--247},
  year    = {1984},
  doi     = {10.2307/2648877}
}

@misc{zhang2026lean,
  author       = {Zhang, Yuanhe and Lee, Jason D. and Liu, Fanghui},
  title        = {{AI4SLT}: Empirical Processes in {Lean 4} for Formal Statistical Learning Theory},
  year         = {2026},
  eprint       = {2602.02285},
  archivePrefix= {arXiv},
  primaryClass = {stat.ML}
}

@misc{zabolotnii2026patp,
  author       = {Zabolotnii, Serhii},
  title        = {Parametrically Adaptive Transition Polynomial: a Signed-Parity Continuous-$\alpha$ Extension of {Kunchenko} Stochastic Polynomials},
  year         = {2026},
  eprint       = {2605.14610},
  archivePrefix= {arXiv},
  primaryClass = {stat.ME}
}

@article{klopotovskyi2026log,
  author  = {Klopotovskyi, Pavlo and Zabolotnii, Serhii},
  title   = {Compressive Logarithmic Generating Elements for Inertialess Filtering of Symmetric Noise},
  journal = {Computer Systems and Information Technologies},
  volume  = {2026},
  number  = {3},
  pages   = {234--243},
  year    = {2026},
  doi     = {10.31891/csit-2026-3-23}
}

@misc{zabolotnii2026cf,
  author       = {Zabolotnii, Serhii},
  title        = {Moment-Free {Kunchenko} Stochastic Polynomials via Empirical Characteristic Function},
  year         = {2026},
  eprint       = {2606.16289},
  archivePrefix= {arXiv},
  primaryClass = {math.ST}
}

@misc{zabolotnii2026paper4,
  author = {Zabolotnii, Serhii},
  title  = {One Truncated Likelihood Expansion: Estimation, Testing, and Classification as a Single Captured-Fraction Functional},
  year   = {2026},
  note   = {Research Square preprint; accepted for publication in Statistical Papers},
  doi    = {10.21203/rs.3.rs-10226291/v1}
}

@article{liebscher2005,
  author  = {Liebscher, Eckhard},
  title   = {A Semiparametric Density Estimator Based on Elliptical Distributions},
  journal = {Journal of Multivariate Analysis},
  volume  = {92},
  number  = {1},
  pages   = {205--225},
  year    = {2005},
  doi     = {10.1016/j.jmva.2003.09.007}
}

@article{battey2014,
  author  = {Battey, Heather and Linton, Oliver},
  title   = {Nonparametric Estimation of Multivariate Elliptic Densities via Finite Mixture Sieves},
  journal = {Journal of Multivariate Analysis},
  volume  = {123},
  pages   = {43--67},
  year    = {2014},
  doi     = {10.1016/j.jmva.2013.08.013}
}

@incollection{ryan2025,
  author    = {Ryan, Victor and Derumigny, Alexis},
  title     = {On the Choice of the Two Tuning Parameters for Nonparametric Estimation of an Elliptical Distribution Generator},
  booktitle = {Stochastic Models, Statistics and Their Applications},
  series    = {Springer Proceedings in Mathematics \& Statistics},
  publisher = {Springer},
  address   = {Cham},
  pages     = {81--124},
  year      = {2025},
  doi       = {10.1007/978-3-031-96015-4_5}
}

@misc{feng2026clustering,
  author       = {Feng, Long and Zhuang, Dan},
  title        = {Semiparametric Elliptical Mixture Clustering for High-Dimensional Data},
  year         = {2026},
  eprint       = {2605.08995},
  archivePrefix= {arXiv},
  primaryClass = {stat.ME}
}

@misc{teng2026clusterwise,
  author       = {Teng, Jen-Chieh and Fan, Sheng-Hsin and Chiang, Chin-Tsang and Huang, Ming-Yueh and Lim, Alvin},
  title        = {Unsupervised Learning Under a General Semiparametric Clusterwise Elliptical Distribution: Efficient Estimation, Optimal Clustering, and Consistent Cluster Selection},
  year         = {2026},
  eprint       = {2604.07917},
  archivePrefix= {arXiv},
  primaryClass = {stat.ME}
}

@article{zhuang2026sslda,
  author  = {Zhuang, Dan and Feng, Long},
  title   = {Spatial Sign Based Direct Sparse Linear Discriminant Analysis for High Dimensional Data},
  journal = {Communications in Statistics -- Simulation and Computation},
  year    = {2026},
  doi     = {10.1080/03610918.2026.2741363}
}

@misc{shen2025ssqda,
  author       = {Shen, Anqing and Feng, Long},
  title        = {A Spatial-Sign Based Direct Approach for High Dimensional Sparse Quadratic Discriminant Analysis},
  year         = {2025},
  eprint       = {2504.11187},
  archivePrefix= {arXiv},
  primaryClass = {stat.ME}
}

\end{document}